\documentclass[10pt,a4paper]{article}
\usepackage[latin1]{inputenc}

\usepackage{mathtools} 
\usepackage{amsfonts}
\usepackage{amssymb}
\usepackage[only,llbracket,rrbracket]{stmaryrd} 
\SetSymbolFont{stmry}{bold}{U}{stmry}{m}{n}

\usepackage{bm}
\usepackage{booktabs}
\usepackage{amsthm}
\usepackage[all]{xy} 	
\usepackage{graphicx} 	
\usepackage{xcolor} 	
\usepackage{listings} 	
\usepackage{pdfpages} 	
\usepackage{float} 		
\usepackage{enumitem}   
\usepackage{hyperref}  

\usepackage{algorithm} 
\usepackage{algpseudocode} 
\algnewcommand\algorithmicinput{\textbf{Input:}}
\algnewcommand\INPUT{\item[\algorithmicinput]}
\algnewcommand\algorithmicoutput{\textbf{Output:}}
\algnewcommand\OUTPUT{\item[\algorithmicoutput]}

\allowdisplaybreaks 
\newcommand{\eps}{{\varepsilon}}
\newcommand{\RR}{{\mathbb{R}}}

\newcommand{\NN}{{\mathbb{N}}}
\newcommand{\ZZ}{{\mathbb{Z}}}
\newcommand{\CC}{{\mathbb{C}}}

\newcommand{\dd}{{\,\mathrm{d}}}
\newcommand{\Fref}[1] {Figure~\ref{#1}}
\newcommand{\Tref}[1]{Table~\ref{#1}}
\def\mydefcal#1{\expandafter\def\csname cal#1\endcsname{\mathcal{#1}}}
\def\mydefallcal#1{\ifx#1\mydefallcal\else\mydefcal#1\expandafter\mydefallcal\fi}
\mydefallcal ABCDEFGHIJKLMNOPQRSTUVWXYZ\mydefallcal
\def\mydefb#1{\expandafter\def\csname bm#1\endcsname{\bm{#1}}}
\def\mydefallb#1{\ifx#1\mydefallb\else\mydefb{#1}\expandafter\mydefallb\fi}
\mydefallb ABCDEFGHIJKLMNOPQRSTUVWXYZabcdefghijklmnopqrstuvwxyz \mydefallb
\def\mydeffrak#1{\expandafter\def\csname frk#1\endcsname{\mathfrak{#1}}}
\def\mydefallfrak#1{\ifx#1\mydefallfrak\else\mydeffrak{#1}\expandafter\mydefallfrak\fi}
\mydefallfrak ABCDEFGHIJKLMNOPQRSTUVWXYZabcdefghijklmnopqrstuvwxyz \mydefallfrak

\renewcommand{\div}{{\operatorname{div}}}

\newcommand{\norm}[2]{\left\| #1 \right\|_{#2}}

\newcommand{\abs}[1]{{\left| #1 \right|}}
\newcommand{\set}[1]{{\left\{ #1  \right\}}}

\newcommand{\oO}[1]{\mathcal{O}\!\left( #1\right )}
\newcommand{\bracket}[1]{\left\langle #1 \right\rangle}

\newcommand{\closure}[1]{{\mkern 1.5mu\overline{\mkern-1.5mu#1\mkern-1.5mu}\mkern 1.5mu}}

\newtheorem{thm}{Theorem}[section]		
\newtheorem{proposition}[thm]{Proposition}

\newtheorem{lemma}[thm]{Lemma}
\newtheorem{definition}[thm]{Definition}
\newtheorem{example}[thm]{Example}
\newtheorem{remark}[thm]{Remark}

\usepackage{euscript}

\newcommand{\bmgamma}{\bm{\gamma}}
\newcommand{\bmrho}{\bm{\rho}}
\newcommand{\bmbeta}{\bm{\beta}}

\newcommand{\train}{\EuScript{S}}
\newcommand{\reg}{\EuScript{R}}

\newcommand{\frdelta}{{\mathfrak \delta}}

\newcommand*\samethanks[1][\value{footnote}]{\footnotemark[#1]}
\title{
       Higher-order Quasi-Monte Carlo Training\\
       of Deep Neural Networks}
\author{M. Longo\thanks{Seminar for Applied Mathematics (SAM), D-MATH \newline
  ETH Z\"urich, R\"amistrasse 101, 
  Z\"urich-8092, Switzerland} \and S. Mishra\samethanks \and T. K. Rusch\samethanks \and Ch. Schwab\samethanks}
\begin{document}
\maketitle
\begin{abstract}
We present a novel algorithmic approach and 
an error analysis leveraging Quasi-Monte Carlo points
for training deep neural network (DNN) surrogates of 
Data-to-Observable (DtO) maps in engineering design.

Our analysis reveals higher-order consistent, 
\emph{deterministic} choices of training points
in the input data space for deep and shallow 
Neural Networks with holomorphic activation
functions such as $\tanh$.

These novel training points are proved to 
facilitate higher-order decay (in terms of the number of training samples) of the underlying generalization error, 
with consistency error bounds that 
are free from the curse of dimensionality 
in the input data space, provided that 
DNN weights in hidden layers satisfy certain summability 
conditions.

We present numerical experiments for DtO maps from elliptic
and parabolic PDEs with uncertain inputs  that confirm the theoretical analysis.
\end{abstract}
\section{Introduction}
\label{sec:Intro}
Many computational problems with PDEs require the 
evaluation of \emph{data-to-observables maps} (DtOs for short)
(functionals, quantities of interest) of the generic form,
\begin{equation}
\label{eq:obs}
\mbox{Given Data}\; y\in Y \mbox{compute observable}~g(y).
\end{equation}
Here, 
the 
observable 
$g: Y \subset {\mathbb R}^d \to {\mathbb R}^{N_{\text{obs}}}$,
a is a function of some prescribed regularity, 
that depends on the solution of an underlying operator equation
subject to input data $y\in Y$.
Such observables arise for instance in 
uncertainty quantification (UQ) of PDEs, 
where $Y$ denotes a data space.
We focus on $Y$ as a bounded subset of euclidean space 
parametrizing input data for the PDE model. 
Other examples include optimal control and design for PDEs, with $Y$
being the control (or design) space. 
Note that this very generic definition of observables in \eqref{eq:obs} 
also includes the solution field by letting $Y$ be the space (space-time) domain. 

Computing \emph{observables} of the form \eqref{eq:obs} requires one 
to numerically approximate PDEs and possibly, 
use quadratures to approximate integrals. 
Given that currently available PDE solvers such as finite element or finite volume methods 
can be computationally expensive, solving \emph{many query} problems such as UQ, 
inverse problems and optimal control (design) for very-high 
dimensional parameter space $Y$ in \eqref{eq:obs}, 
that require a large number of calls to the underlying PDE solver, 
can be prohibitively expensive. 

\emph{Surrogate models} \cite{SUR} provide a possible pathway for 
reducing the computational cost of such \emph{many query} problems for PDEs. 
Surrogates such as reduced order models \cite{AQAMFNRB} and Gaussian process regression \cite{GPRbook}, 
build a \emph{surrogate approximation}
$\hat{g}: Y \mapsto {\mathbb R}^{N_{\text{obs}}}$ such that $\hat{g} \approx g$ in a suitable sense. 
As long as the surrogate is sufficiently accurate and the cost of 
evaluating the surrogate is 
significantly less than the cost of numerically evaluating 
the underlying DtO map $g$ \eqref{eq:obs} with the same level of fidelity,
one can expect the surrogate model to be more computationally efficient than 
so-called ``high-fidelity'' PDE solvers. These deliver, for example
by discretization of the PDE with discretization parameter $\frdelta\in (0,1]$,
a one-parameter family of approximate forward maps $\{g_\frdelta\}_{0<\frdelta\leq 1}$,
which is assumed to be consistent with $g$ in the sense that 
\begin{equation}\label{eq:gdeltcnsist}
\lim_{\frdelta\downarrow 0} g_\frdelta = g \quad\mbox{uniformly with respect to} \;\; Y\;.
\end{equation}
Here, the discretization parameter $\frdelta$ could, e.g., be a stepsize $\Delta t>0$,
a FE/FD meshwidth $h>0$ or the reciprocal of a spectral order.

Building surrogates $\hat{g}$ to $g$ with certified fidelity uniform with respect to the
set of input data $Y$ can be challenging. 
Accordingly, during the past decade computational science and engineering 
has witnessed the arrival of mathematical and computational frameworks 
aiming at generating such surrogates computationally, and to quantify
the corresponding emulation error mathematically. 
These frameworks go by the name of 
``Reduced Basis (RB) methods'' or ``Model Order Reduction (MOR) techniques''. 
They aim at \emph{computational determination of low-dimensional subspaces} $X_N$
of the vector space $X$ containing the response of the PDE model of interest.
We refer to the surveys \cite{AQAMFNRB,HestRozStamRB} and the references there
for details and theory. 
In MOR and RB, the phase of building $\hat{g}$ is usually referred to as 
``offline phase'' and is a) usually quite costly and b) 
is based on executing certain greedy searches on 
numerical approximations $g_\frdelta$ of the PDE of interest.

Deep neural networks (DNNs) (e.g. \cite{DLbook}; specifically, here the term 
``DNN'' will denote a so-called feed-forward NN)
are concatenated, repeated compositions 
of affine maps and scalar non-linear activation functions.
In recent years, DNNs emerged as another 
powerful tool in computational science with well-documented success in a variety of
tasks such as image classification, text and speech recognition, 
robotics and protein folding \cite{DLnat}. 
Their mathematical structure allows the interpretation of 
MOR and RB as particular instances 
(see, e.g., \cite{RegDedAQ_MORODEDNN} for a 
development of this point of view in the context of parametric dynamical systems).
Given their \emph{universality}, i.e., the
ability to approximate (``\emph{express}'' in the terminology of the deep learning community) 
large classes of functions 
(e.g. \cite{PinkusActa} and the references there),
and their high approximation rates on regular maps 
(e.g. \cite{BGKP2019,OSZ19_2771,OPS19_2738} and the references there),
DNNs are increasingly being used in various contexts 
for the numerical approximation of  DtO maps for PDEs \cite{KAR1,E1,E2}.

In particular, recent papers such as \cite{LMR1,LMM1,MR1,LMRP1} 
have proposed using DNNs for building surrogates for \emph{observables} of PDEs 
and applying these surrogates to accelerate UQ for PDEs \cite{LMR1,LMM1} 
and PDE constrained optimization \cite{LMRP1}. 
These articles use DNNs within the paradigm of \emph{supervised learning} i.e. select a \emph{training set} 
$\train \subset Y$ and use a (stochastic) gradient descent algorithm 
to find tuning parameters (weights and biases) 
that provide the smallest mismatch between the underlying DtO map $g$ 
and the resulting neural network on this training set. 

It is standard in machine learning \cite{MLbook}
to choose independent and identically distributed 
random points in $Y$ to constitute the training set. 
However, 
as pointed out in \cite{LMR1,MR1} and references therein, 
the so-called \emph{generalization gap}, 
i.e. the difference between the \emph{generalization error} or \emph{population risk} 
(see \eqref{eq:gerr} for a precise definition) and 
the computable \emph{training error} or \emph{empirical risk} for \emph{trained} DNNs 
scales, at best, as $1/\sqrt{N}$ (in the root mean square sense), 
with $N = \#(\train)$ being the number of training points (samples).
We refer again to \cite{MLbook} and references therein 
for sharper estimates on the generalization gap. 

Given this slow decay of generalization error in terms of the number $N$
of i.i.d random training samples,
one possibly needs a large number of samples to achieve 
a desired level of DNN fidelity.
In emulation of DtOs from PDEs,
training data is generated by $N$-fold calling the underlying PDE solver 
(with discretization error $|g - g_\frdelta|$ well below the DNN target fidelity).

Thus,
a potentially prohibitively large number $N$ 
of calls to the `high-fidelity' forward PDE solver 
could preclude efficient training and surrogate modeling, 
see, e.g. \cite{LMR1} and references therein. 
In this reasoning, 
we must distinguish between the {\bf exact DtO map} $g$ 
(which is, generally, not numerically accessible) 
and its 
{\bf numerical approximations} $\{ g_\frdelta : 0 < \frdelta \leq 1 \}$.
For our results to hold for the exact forward map $g$, 
we require the discretization error $\delta$ in the numerical approximation $g_\delta$ 
of the DtO map $g$ to be uniformly smaller than the target DNN emulation fidelity 
$\hat{\eps} > 0$ of the DNN $\hat{g}$ approximating $g$.
I.e., we require 
\begin{equation}\label{eq:hateps}
\sup_{y\in Y} | g(y) - g_\delta(y)| \leq \delta \stackrel{!}{\leq} \hat{\eps}: = \sup_{y\in Y}| g(y) - \hat{g}(y)| \;.
\end{equation}
In order to alleviate prohibitive DNN training cost, 
one could consider more sophisticated training designs $\train$.
The authors of \cite{LMR1,MR1} propose using \emph{low-discrepancy sequences}, 
such as Sobol' and Halton sequences used 
in \emph{quasi-Monte Carlo} (QMC) quadrature algorithms \cite{CAF1}, as training points. 
In \cite{MR1}, the authors prove that as long as the underlying map \eqref{eq:obs} is of bounded Hardy-Krause variation, 
one can prove that the generalization gap for supervised DNN with QMC training points, 
decays (upto a logarithmic correction) as $N^{-1}$ i.e. linearly in the number of training samples. 
Furthermore, 
these training points lead to deterministic bounds on the generalization gap 
that are inherently more robust and easier to verify than 
probabilistic root mean square bounds with random training points.  
Numerical examples, presented in \cite{LMR1} and \cite{MR1} 
demonstrate the increased efficiency of using QMC points for training, 
when compared to random points.

However, as is well known, 
the logarithmic correction stemming from the Koksma-Hlawka inequality for QMC, 
\emph{depends exponentially} 
on the underlying dimension $d$ of the parameter space $Y$. 
Consequently, 
the proposed deep learning algorithm with (for example) sets $\train$ 
chosen as Sobol' training points, 
suffers from the \emph{curse of dimensionality}. 
This is also demonstrated in numerical experiments in \cite{MR1} 
where the deep learning algorithm based on QMC training points, 
outperforms the one with random training points, only for problems in moderately high dimensions. 
It is natural to ask if one can find training sets $\train$ for DNNs 
which overcome this curse of dimensionality, 
while still possessing a faster rate of decay than the use of i.i.d random training points. 
In particular, if one can find training  point designs $\train$
which ensure higher than linear rate of decay of the generalization gap, 
independent of high parameter space dimension.

It turns out that recent developments in \emph{Quasi-Monte Carlo integration} 
algorithms, namely the design of \emph{higher order QMC} (HoQMC) rules 
\cite{DLS16,DGLS19,GS16} can provide positive answers to the above questions. 
In particular, in the context of numerical integration, 
these rules lead to a dimension-independent, superlinear decay of the quadrature error 
as long as the integrand appearing in the loss function 
is \emph{holomorphic} with holomorphy domains quantified in terms 
of the coordinate and the integrand dimension; 
see, e.g., \cite{CCS15,DLS16}. 

The loss function being (an integral of) a difference between
the map $y\mapsto g(y)$ and a DNN surrogate $y\mapsto\hat{g}(y)$,
this entails holomorphy requirements of both, the DtO map, as well
as of the DNN surrogates. We
point out that a large number of PDEs, 
particularly of the elliptic and parabolic type possess solutions 
(and observables)  whose DtO maps are holomorphic in the parameter space. 
Moreover, in \cite{CCS15} the authors introduced two sufficient criteria 
that ensure holomorphy in a variety of cases, 
including UQ for non-linear PDEs and shape holomorphy 
(e.g. \cite{CCS15,CSZ18_2319,HS19_847,AJSZ20_2734} and the references here).

Motivated by the higher-order QMC rules in \cite{DLS16,DGLS19,GS16}, 
in this article we make the following \emph{contributions}:
\begin{itemize}
\item 
We develop several \emph{novel DNN training strategies}, 
based on the use of \emph{deterministic, higher order Quasi-Monte Carlo (QMC) point designs}
as DNN training points, 
in order to emulate Data-to-Observables maps 
for systems governed by parametric PDEs.
\item 
We prove, 
under \emph{quantified holomorphy hypotheses 
on the DtO map to be emulated by the DNN
and 
on the scalar activation function of the DNN}, 
that \emph{for any input dimension} $d$, 
the generalization gap of the resulting trained DNN decays 
superlinearly with respect to the number $N$ of training points, 
and independent of the dimension of data space.
I.e., we prove a bound 
$O(N^{-\alpha})$ with $\alpha \geq 2$ and $N$ being the number of training points, 
with the constant implied in $O( )$ being independent of the dimension $d$ of
the input parameter domain. 
Thus, 
the proposed deep learning algorithm can achieve significantly lower errors
than the one based on random i.i.d training points, 
while still being free of the curse of dimensionality.
\item 
We present a suite of numerical experiments for data-to-observables which arise from parametric
PDEs with uncertain input data to illustrate the theory.
We also show numerical experiments which strongly indicate that several hypotheses on holomorphy and
sparsity in our results appear to be necessary, while others seem to be 
artifacts of our proofs based on complex-variable techniques.
\end{itemize} 
The rest of the paper is organized as follows. 
The deep learning algorithm is presented in Section \ref{sec:2} and is analyzed in Section \ref{sec:3}. 
Several illustrative
numerical experiments are presented in Section \ref{sec:4} and the 
contributions of the current article and possible extensions are discussed in Section \ref{sec:5}. 
\section{ Deep Learning on higher-order Quasi-Monte Carlo training points}
\label{sec:2}
In the present section we briefly recapitulate elements from Quasi-Monte Carlo integration,
as they pertain to the proposed higher-order lattice integration schemes which we subsequently
use for defining the DNN loss function.
In Section \ref{sec:DefDNN} we define the architectures of DNNs that we consider.
Section \ref{sec:lf} then introduces the computable loss functions which we use in DNN training.
and outline our proposed 
``Deep learning with Higher-order Quasi-Monte Carlo points (DL-HoQMC)'' algorithm.
\subsection{Higher-order Quasi-Monte Carlo rules}
\label{sec:HoQMC}
We consider two classes of higher order QMC quadrature rules in this article.
Either class of rules is derived from first order digital nets construction of Polynomial lattices, 
as originally introduced by Niederreiter in \cite{N92}. 
We briefly recapitulate the essentials.
In the following, let $b \ge 2$ be a prime number, 
$\mathbb{F}_b$ be the finite field with $b$ elements, 
$\mathbb{F}_b[x]$ be the set of all polynomials over $\mathbb{F}_b$ 
and 
$\mathbb{F}_b((x^{-1}))$ be the set of all formal Laurent series of the form
$\sum_{i=t}^\infty a_i x^{-i}$, $t \in \mathbb{Z}$, $a_i$ in $\mathbb{F}_b$.  

We can identify an integer $0 \le n < b^m$ given by the $b$-adic expansion $n = n_0 + n_1 b + \cdots + n_{m-1} b^{m-1}$ and $n_0, \ldots, n_{m-1} \in \set{0, 1, \ldots, b-1}$, with its corresponding polynomial $n(x) \in \mathbb{F}_b[x]$ given by $n(x) = n_0 + n_1 x + \cdots + n_{m-1} x^{m-1}$, where we now view $n_0, \ldots, n_{m-1} $ as elements of $ \mathbb{F}_b$.

\begin{definition} [Polynomial lattice rule]
Let $m \ge 2$ be an integer and $p \in \mathbb{F}_b[x]$ be a polynomial with $\deg(p) = m$. Let $\bm{q} = (q_1, \ldots, q_d)$ be a vector of polynomials over $\mathbb{F}_b$ with degree $\deg{q_j} < m$. We define the map $v_m: \mathbb{F}_b((x^{-1})) \to [0,1)$ by
\begin{equation*}
v_m \left( \sum_{i=t}^\infty a_i x^{-i} \right) = \sum_{i = \max(1, t)}^m a_i b^{-i},
\end{equation*}
for $ t \in \ZZ $.
For $0 \le n < b^m$, we put
\begin{equation*}
y_n = \left(v_m\left(\frac{n(x) q_1(x)}{p(x)} \right), \ldots, v_m\left(\frac{n(x) q_d(x)}{p(x)} \right) \right) \in [0,1)^d.
\end{equation*}
Then the set $ P_m(\bmq,p) = \set{y_0, y_1, \ldots, y_{b^m-1}}$ is called a polynomial lattice point set and a quadrature rule $ Q_{b^m,d} $,  using this point set is called a polynomial lattice rule.
\end{definition}

The following class of higher order lattice rule was proposed in \cite{DGY19} and developed in \cite{DLS20_902}. It mainly relies on an asymptotic expansion of the quadrature error that allows to apply Richardson extrapolation to the sequence of quadrature rules $ (Q_{b^m,d})_{m \in \NN} $, when applied to integrands of sufficient regularity. 

\begin{definition}[Extrapolated Polynomial lattice] \label{def:EPL}
Let 
$ \alpha \ge 2 $ be a natural number 
and 
$ Q_{b^{m - \tau + 1},d}, \tau \in\set{ 1,\dots, \alpha} $ 
be a set of Polynomial lattice rules with associated lattice point sets 
$ P_{m - \tau + 1}(\bmq_{m - \tau + 1},p_{m - \tau + 1}) $ respectively.

We define, for suitable coefficients $ a_{\tau}^{(\alpha)} $  defined in \cite[Lemma 2.9]{DGY19}, 
the quadrature rule
\begin{equation*}
Q_{b^m,d}^{(\alpha)} = \sum_{\tau= 1}^{\alpha} a_{\tau}^{(\alpha)} Q_{b^{m - \tau + 1},d}
\end{equation*}
and we call it \emph{Extrapolated Polynomial Lattice} (EPL) rule of order 
$ \alpha $ and lattice cardinality $ N = b^{m -\alpha + 1}+ \ldots + b^m $.
\end{definition}
An alternative definition of higher order QMC rules is based on so-called 
\emph{digit interlacing polynomial lattices}, 
as explained in the following definition, \cite{DKLNS14,DGLS19,DLS16,GS16}.
\begin{definition}[Interlaced Polynomial lattice]
Let $ \alpha \ge 2 $ be a natural number and let
$ \calD_{\alpha}\colon [0,1)^{d\alpha} \to [0,1)^{d}$ 
be defined by 
\begin{equation*}
\calD_{\alpha}(x_{1},\ldots,x_{d\alpha}) 
= 
(\calD_{\alpha}(x_{1},\ldots,x_{\alpha}),\ldots,\calD_{\alpha}(x_{(d-1)\alpha + 1},\ldots,x_{d\alpha})),
\end{equation*} 
and satisfying for any 
$x = \left (\sum_{i=1}^{\infty}x_{1,i}b^{-i},\ldots,\sum_{i=1}^{\infty}x_{\alpha,i}b^{-i}\right )\in [0,1)^{\alpha}$, 
with $ x_{\tau,i}\in \mathbb{F}_b , \ \forall \tau=1,\ldots,\alpha, \forall i \in \NN$ 
and such that $ (x_{\tau,i})_{i>i_0} $ is not constant equal to $  b-1 $ after any index $ i_0 \in\NN $, 
\begin{equation*}
    \calD_{\alpha}(x) = \sum_{i=1}^{\infty} \sum_{\tau=1}^{\alpha} x_{\tau,i} b^{-(\alpha (i-1) + \tau)} \in [0,1).
\end{equation*} 
Then, a quadrature rule using $ \calD_{\alpha}(P_m((q_1, \ldots, q_{\alpha d}),p)) $ as point set is called \emph{Interlaced Polynomial Lattice} (IPL) rule of order $ \alpha $ and cardinality $ N = b^m $.
\end{definition}
In the next sections we will always work with the natural choice of basis $ b=2 $ 
so that digit operations become bit operations in the numerical computations.
Conversely to classical QMC point sets as Sobol' and Halton sequences, 
EPL and IPL are known to achieve dimension-independent error bounds (e.g. \cite{DKLNS14,DLS16,DGLS19,DGY19}). 
The main reason for such improvement is that their construction can be done in a problem-dependent manner,  
exploiting the varying importance of the individual components in the vector $ y = (y_1,\ldots, y_d) \in Y $. 
For this purpose we define a set of positive \emph{weights} 
$ \bmgamma := (\gamma_{\frku})_{\frku \subseteq \set{1,\ldots,d}} $, 
that quantify the relative importance of the variables in the set $ \frku \subseteq \set{1,\ldots,d} $.
In our discussion we will need QMC weights 
in SPOD\footnote{SPOD: ``Smoothness-driven, Product and Order Dependent'', see \cite{DKLNS14}.}
form,
\begin{equation}
\gamma_{\frku} := \sum_{\bm{\nu} \in \set{1:\alpha}^{\abs{\frku}}} \abs{\bm{\nu}}! \prod_{j \in \frku} \left ( 2^{\delta(\nu_j, \alpha )}  \beta_{j}^{\nu_j}\right ), \qquad (\beta_{j})_j \in \ell^{1}(\NN)
\end{equation}
that allow for a wider class of applications compared to product weights $ \gamma_{\frku}:=\prod_{j \in \frku} \beta_j, \ \frku \subseteq \set{1,\ldots,d} $. The weights play a key role in the \emph{Component-By-Component} (CBC) construction of a generating vector $ \bmq $ of a polynomial lattice. The details of the CBC construction and their fast version using FFT can be found in \cite{DKLNS14} for IPL rules and in \cite{DLS20_902} for EPL rules.

The CBC construction for EPL rules proves to be slightly cheaper in terms of operations: $ \oO{(\alpha + d)N \log N + \alpha^2 d^2 N} $ for EPL rules versus $ \oO{\alpha d N \log N + \alpha^2 d^2 N} $ for IPL, both requiring $ \oO{\alpha d N} $ memory. On the other hand, IPL rules seem to give slightly better outcomes as shown in \cite{DGY19}, leading to an overall comparable performance. One advantage of EPL over IPL, is that they allow for computable, asympotically exact, a-posteriori error estimates \cite{DLS20_902}. 

We will consider either class of point sets as the training sets of our deep learning algorithm that we describe below. 

\subsection{Deep Neural networks}
\label{sec:DefDNN}
We consider the following form of deep neural networks (DNNs) in this paper. Let $\sigma\colon \RR^k \to \RR^k $ 
be a nonlinear activation function for $ k \in \NN $, $ L \in \NN $
and a collection of weights 
$ W^{(\ell)} \in \RR^{d_{\ell}\times d_{\ell-1}} $, 
and biases
$ b^{(\ell)} \in \RR^{d_{\ell}} $ for $ \ell = 1, \ldots,L $. We shall refer to the integer $L \geq 1$ as \emph{depth of the DNN} and 
we denote the layer widths tuple $ {\mathfrak d} := \{ d_0,d_1,...,d_L \} \in \NN^{L+1} $ as \emph{the architecture of the DNN};
that is, $ d_{0} = d \in \NN $ is the input dimension and 
$ d_{L} = N_{\text{obs}}$ denotes the output dimension in the definition of the underlying observable \eqref{eq:obs}.
We collect the set of parameters of the DNN in 
\begin{equation}\label{eq:DNNpara}
\Theta := \set{(W^{(\ell)},b^{(\ell)}) \in \RR^{d_{\ell}\times d_{\ell-1}} \times \RR^{d_{\ell}} \colon \ell = 1,\ldots, L }.
\end{equation}
Notice that $\Theta$ depends on the sequence ${\mathfrak d}$ 
which we do not indicate notationally. 
Also,
we do not impose any further conditions (such as sparsity, clipping or quantization)
on the weight matrices $W^{(\ell)}$ or on the bias vectors $b^{(\ell)}$ in \eqref{eq:DNNpara}.

For any $ (W,b) \in \RR^{d\times d'} \times \RR^{d} $,  $ d,d' \in \NN $,
let  $ f_{W,b} \colon \RR^{d'} \to \RR^{d} $ denote nonlinear map which is defined by
\begin{equation}\label{def:fWb}
f_{W,b}(y) := \sigma(Wy+ b)\;.
\end{equation}
For $\theta\in \Theta$,
define the Neural network map 
$\phi_{\theta}^L: \RR^{d} \to \RR^{d_{L}}$ 
as
\begin{equation} \label{def:dnn}
\phi_{\theta}^{L}(y):= W^{(L)}(f_{W^{(L - 1)},b^{(L - 1)}} \circ \cdots \circ f_{W^{(1)},b^{(1)}}(y)) + b^{(L)}
\;.
\end{equation}
We observe that the form \eqref{def:dnn} of the neural network corresponds to that
of a fully-connected feed-forward multi-layer perceptron \cite{DLbook}. 
This form is very general and other specific types of
neural networks, such as convolutional neural networks (CNNs), or sparsely connected DNNs,
can be realized by imposing constraints on the structure of the weight matrices in $\Theta$ which are used
in \eqref{def:dnn}. 
\subsection{Loss functions}
\label{sec:lf}
As we adopt the paradigm of supervised learning \cite{DLbook}, 
DNNs of the form \eqref{def:dnn} will be \emph{trained} to determine parameters $\theta\in \Theta$
in concrete applications.
I.e., a parameter vector $\theta$ has to be found numerically 
such that the mismatch between the
ground truth (underlying map $g$ \eqref{eq:obs}) and the DNN is 
minimized over the \emph{training set}. 
To this end, we define suitable \emph{loss functions to quantify the mismatch} 
between the map $g$ and its DNN surrogate.

In this context, 
we define for any
suitable polynomial lattice point (or IPL) set $P_m(\bmq, p)$ of cardinality $ 2^m$
\begin{equation}
\label{eq:plf}
\tilde{\calE}_{T,m}(\theta):= 
\left (\frac{1}{2^m} \sum_{y \in P_m(\bmq,p)} \abs{g(y) - \phi_{\theta}^{L}(y)}^{2} \right )^{1/2} , \quad \theta \in \Theta.
\end{equation}

Next, we differentiate between two cases. 
First, for IPL training points, 
we consider the following partial loss function, 
\begin{equation}
\label{eq:ipllf}
J(\theta) := \left(\tilde{\calE}_{T,m}(\theta)\right)^2, \quad \theta \in \Theta.
\end{equation}

Second, for the EPL training points, 
we need to work with suitable extrapolation of the partial loss functions
\begin{equation} \label{eq:epllfalpha}
\calE_{T}^{(\alpha)}(\theta) 
:= 
\left | \sum_{\tau= 1}^{\alpha} a_{\tau}^{(\alpha)} \left (\tilde{\calE}_{T,m - \tau + 1}(\theta)\right )^2 \right |^{1/2},
\theta \in \Theta,
\end{equation}
with the coefficients $ a_{\tau}^{(\alpha)} $ as in Definition \ref{def:EPL}. 
For notational simplicity, 
we confine ourselves to $\alpha = 2$, that reads $ a_{1}^{(2)} = 2, a_{2}^{(2)} = -1 $ so that
\begin{equation}
\label{eq:epllf1}
\calE_{T}^{(2)}(\theta) =  \abs{2\tilde{\calE}_{T,m}^{2}(\theta) - \tilde{\calE}_{T,m-1}^{2}(\theta)}^{1/2}.
\end{equation}

We hasten to add, however, that all results generalize verbatim to higher digit interlacing order $\alpha > 2$, 
resp. to higher extrapolation orders.
While the use of IPLs naturally results in positive coefficients in the loss function,
EPLs, being obtained by Richardson type extrapolation formulas, 
and thus involve alternating signs of coefficients.
In numerical DNN training,
solving the optimization problem with alternating sign linear combinations
can be computationally delicate. 
Hence, 
we define the loss function for the EPL training points 
by the following upper bound on \eqref{eq:epllf1},
\begin{equation}
\label{eq:epllf}
J(\theta) =  2 \tilde{\calE}_{T,m}^2(\theta) + \tilde{\calE}_{T,m-1}^2(\theta).
\end{equation}
Note that the loss function \eqref{eq:epllf} 
can be readily generalized for any $\alpha > 2$, 
replacing the coefficients $ a_{\tau}^{(\alpha)}$ 
in \eqref{eq:epllfalpha} by their absolute values.

The goal of the training process in supervised learning is to find the parameter vector $\theta$, 
for which the loss functions \eqref{eq:ipllf} or \eqref{eq:epllf} are minimized. 
It is common in machine learning \cite{DLbook} to regularize the minimization problem for the loss function, 
i.e. we seek to find
\begin{equation}
\label{eq:lf}
\theta^{\ast} = {\rm arg}\min\limits_{\theta \in \Theta} \left(J(\theta) + \lambda \reg(\theta) \right).
\end{equation}  
Here, $J$ is defined by either \eqref{eq:ipllf} (for IPL training points) or \eqref{eq:epllf} (for EPL training points)  and $\reg$ is a \emph{regularization} (penalization) term. A popular choice is to set  $\reg(\theta) = \|\theta_W\|^q_q$ , with $\theta_W$ denoting the concatenated vector of all weights in \eqref{def:dnn} and either $q=1$ (to induce sparsity) or $q=2$. The parameter $0 \leq \lambda \ll 1$ balances the regularization term with the actual loss $J$. 

The above minimization problem amounts to finding a minimum of a possibly non-convex function over a very high-dimensional parameter space. We follow standard practice in machine learning by either (approximately) solving \eqref{eq:lf} with a full-batch gradient descent algorithm or variants of mini-batch stochastic gradient descent (SGD) algorithms such as ADAM \cite{adam}. 

For notational simplicity, we denote the (approximate, local) minimum weight vector in \eqref{eq:lf} as $\theta^{\ast}$ and the underlying deep neural network $\phi_{\theta^{\ast}}^L$ will be our neural network surrogate for the underlying map $g$. The proposed algorithm for computing this neural network is summarized below. \\\\
{\bf Deep learning with Higher-order Quasi-Monte Carlo points (DL-HoQMC)} 
\begin{itemize}
\item [{\bf Inputs}:] 
Underlying map $g$ \eqref{eq:obs}, 
higher-order QMC training points such as EPL or IPL points, 
hyperparameters and architecture of neural network \eqref{def:dnn} with depth $L$
\item [{\bf Goal}:] 
Find neural network $\phi_{\theta^{\ast}}^L$ for approximating the underlying map $g$. 
\item [{\bf Step $1$}:] 
Choose the training set $\train$ either as IPL or EPL QMC points. 
Evaluate $g(y)$ for all $y \in \train$ by a suitable numerical method. 
\item [{\bf Step $2$}:] 
For an initial value of the weight vector $\overline{\theta}$, 
evaluate the neural network $\phi_{\overline{\theta}}^L$ \eqref{def:dnn}, 
the loss function \eqref{eq:lf} and its gradients to initialize the
(stochastic) gradient descent algorithm.
\item [{\bf Step $3$}:] 
Run a stochastic gradient descent algorithm till an approximate local minimum $\theta^{\ast}$ of \eqref{eq:lf} is reached. 
The map $\phi^L_{\theta^{\ast}}$ 
is the desired neural network approximating the map $g$.
\end{itemize}
\section{Analysis of the DL-HoQMC algorithm}
\label{sec:3}
The objective of our analysis of the DL-HoQMC algorithm would be to estimate the so-called \emph{generalization error} 
of this algorithm which is defined as $\calE_{G} = \calE_{G}(\theta^{\ast})$, with 
\begin{equation}
\label{eq:gerr}
\calE_{G}(\theta) = \left (\int_{Y} \abs{g(y) - \phi_{\theta}^{L}(y)}^{2} dy \right )^{1/2}, \quad \theta \in \Theta.
\end{equation}
Throughout the rest of this work we adhere to the convention that the input data is appropriately scaled to the box $ Y:=[0,1]^d $ of Lebesgue measure $ \abs{Y}=1 $. Hence, the expressions \eqref{eq:plf}, \eqref{eq:epllfalpha} are QMC quadrature approximations of $\calE_{G}(\theta)$.

As is customary in machine learning \cite{MLbook,AR1}, we will estimate the generalization error in terms of \emph{computable} training errors such as \eqref{eq:plf} for the IPL training points and \eqref{eq:epllfalpha} for the EPL training points. The key is to realize that the training errors \eqref{eq:plf} and \eqref{eq:epllfalpha} are the QMC quadratures for the integral in \eqref{eq:gerr} defining generalization error. 
Thus, 
higher-order QMC approximation results  (e.g. \cite{DKSActa} and the references there and, for the 
presently proposed QMC integrations, \cite{DLS16,DKLNS14})
can be brought into play to estimate the so-called \emph{generalization gap} 
i.e. difference between quadrature error and computable training errors.

Although the DL-HoQMC algorithm can be applied for approximating any underlying map $g$, 
it is clear from the higher-order QMC theory 
that \emph{dimension independent higher-order approximation} results can only be
obtained for integrands in loss functions 
that exhibit sufficient regularity  with explicit, quantified dependence on the 
coordinate dimension.
Our starting point in determining the appropriate function class for the underlying map as well as 
the approximating neural network is the \emph{weighted unanchored Sobolev space} 
$\calW_{d,\alpha,\bmgamma, 1,\infty}$, 
which is defined by the set of integrand functions 
$ F \in C^{\infty}([0,1]^d)$ 
equipped with the norm
\begin{equation}\label{def:wNorm}
\norm{F}{d,\alpha,\bmgamma, 1,\infty} 
:= 
\sup_{\frku \in \set{1:d}} \frac{1}{\gamma_{\frku}} 
\sum_{\frkv \subseteq \frku} \sum_{\bm{\nu}_{\frku \setminus \frkv} \in \set{1:\alpha}^{\abs{\frku \setminus \frkv}}}
\int_{[0,1]^{\abs{\frkv}}} \abs{\int_{[0,1]^{d - \abs{\frkv}}} \partial_{y}^{(\bm{\nu}_{\frku \setminus \frkv}, \alpha_{\frkv})} F(y) \dd y_{\frkv^c}  } \dd y_{\frkv},
\end{equation}
for a set of weights $ \bmgamma $ to be determined.
 Inspecting \eqref{def:wNorm}, it transpires that an error analysis will require
estimates of higher order derivatives of input-output maps of DNNs in terms of 
(bounds on) NN weights, biases and activations. 
With the focus on deep NNs, we found the use of the multivariate chain rule 
in bounding $\norm{\hat{g}}{d,\alpha,\bmgamma, 1,\infty}$ prohibitive
due to the compositional structure of DNNs.
Instead, we opt on using complex 
variable techniques based on quantified holomorphy to that end. 
Being derivative-free and preserved under composition,
it appears naturally adapted to the analysis of DNNs.
Using complex variable techniques mandates
\emph{holomorphic activation functions} in the DNNs, though.

Our goal is to find conditions on the underlying map $g$ 
and on the weights $ W^{(\ell)} \in \RR^{d_{\ell}\times d_{\ell-1}} $ 
and biases $ b^{(\ell)}\in \RR^{d_{\ell}}, \ell \in 1\ldots L$ of the DNN \eqref{def:dnn}, 
which ensure that the integrand in \eqref{eq:gerr} i.e. 
$\abs{g - \phi_{\theta}^{L}}^{2} \in \calW_{d,\alpha,\bmgamma, 1,\infty}$ 
uniformly with respect to the input dimension $d$.  
A sufficient condition relies on the concept of 
$(\bmbeta,p,\eps) $-holomorphy with $ 0 < p < 1 $ in the sense of \cite[Theorem 3.1]{DLS16}.  
We define this concept below. 
\begin{definition}[$ (\bmbeta,p,\eps)$-holomorphy on polytubes] 
\label{def:bpepsholo}\cite{DLS16,CCS15}
Let $ X $ be a Banach space over $ \CC $, $ \eps > 0  $, $ 0 < p < 1 $ and 
let $ \bmbeta \in \ell^p(\NN)$ be a non-negative sequence in $ \RR $. 
Define the \emph{polytubes} $ \calT_{\bmrho} = \bigtimes_{j\ge 1} \calT_{\rho_j} $,
where 

\begin{equation*}
\calT_{\rho_j} = \set{z \in \CC \colon \operatorname{dist}(z,[-1,1]) < \rho_j - 1} \qquad \rho_j \in (1,\infty).
\end{equation*}
\begin{subequations}\label{bpepseq}
We say that a $ \bmrho = (\rho_{j})_{j \in \NN} \in (1,\infty)^{\NN} $ is $  (\bmbeta,\eps) $-admissible, if there holds
\begin{equation}
\sum_{j \ge 1} (\rho_{j} - 1)\beta_{j} \le \eps \qquad \rho_j \in (1,\infty).
\end{equation}
 Define $U := [-1,1]^{\NN} \subset \calT_{\bmrho}$.
A map  $\phi\colon U \to X$ is called $(\bmbeta,p,\eps)$-holomorphic if:
\begin{enumerate}
\item 
for all $ \bmrho $ that is $ (\bmbeta,\eps) $-admissible,  
$ \phi $ admits holomorphic extension with respect to each variable
on the polytube $ \calT_{\bmrho}$, and
\item 
there exists a family of open sets $ \calO_{\bmrho} \supset \closure{\calT_{\bmrho}}$ 
and a constant $ C_{\eps} > 0 $  
independent of $ \bmrho$ such that there holds the uniform bound 
\begin{equation} 
\sup_{z \in \calO_{\bmrho}} \norm{\phi(z)}{X} \le C_{\eps}.
\end{equation}
\end{enumerate}  
\end{subequations}
\end{definition}
Definition \ref{def:bpepsholo} will be, in our case, 
accommodated by mapping $ U = [-1,1]^{\NN} $ to the domain $ [0,1]^{\NN}$, 
by means of an affine (in particular holomorphic) 
change of variables 
$ \xi \colon y \mapsto (y+1)/2$. 
Moreover, the infinite dimensional parameter space is an artifact 
to obtain expression rate bounds that are free from the curse of dimensionality. 
We define the underlying map $ \bar{g} \circ \xi : U \to \RR^{N_{\text{obs}}}$ 
as a function with infinitely many parameters and we view the quantity of interest as its truncated version via
\begin{equation}\label{bargtrunc}
g \circ \xi((y_1,\ldots,y_d))=\bar{g} \circ \xi((y_1,\ldots,y_d,0,0,\ldots)), \quad (y_1,\ldots,y_d) \in [-1,1]^d\;.
\end{equation} 
That is, we anchor the parameters after $ d $ to the center of their domain.
Note that holomorphy, thus analyticity of the integrand, also 
allows to recover super-polynomial convergence (with respect to $ N $) 
of the training error to the generalization error by training, 
for example,  on tensorized Gauss-points 
$\Gamma_{n,d} = \set{(y_1, \ldots, y_d) \in Y = [0,1]^d \colon y_i \in \Gamma_{n,1}, \ \forall i }$, 
where $ \Gamma_{n,1} $ denote the Gauss quadrature points in the interval $[0,1]$. 
However, the implied convergence of $ \calO(\exp(-r N^{1/d}))$, 
with $r>0$ independent of $N$ and of $d$,
deteriorates quickly as the parameter dimension $d$ 
increases and results in practically infeasible training designs,
for even moderate values of $d$. 
QMC training designs $\train$ only afford algebraic rates of convergence
     in terms of $N = \#(\train)$ which are free from the curse of dimensionality, though.

In order to investigate the holomorphy of neural networks \eqref{def:dnn}, 
we formally extend $ f_{W,b} $ in \eqref{def:fWb} for complex sequences  
$ z = y + i\eta \in \CC^{\NN}$, $ y,\eta \in \RR^{\NN} $ 
and semiinfinite (``sequence'') arrays 
$ W = \big (W_{ij}\big)_{\substack{1 \le i \le d \\1 \le j < \infty}} \in \RR^{d \times \NN} $, 
and we write
\begin{equation*}
	f_{W,b}(z) = \sigma(W z+ b).
\end{equation*}
We remark that we retain the network parameters $ W,b $ real valued. 
Verifications of the $ (\bmbeta,p,\eps) $-holomorphy of DNNs is 
considered in the following two sections.
\subsection{Quantified Holomorphy of Shallow Neural Networks}
\label{sec:HolShalNNs}
We start the study of holomorphy of DNNs by first considering 
the shallow ($ L = 2 $ in \eqref{def:dnn})  but possibly wide neural network. 
Let $ R > 0 $, we use the notation 
\begin{equation} \label{def:strip}
\calS_{R} = \set{z \in \CC \colon \abs{\Im(z)} < R }
\end{equation} 
for the \emph{strip} of width $ 2R $ around the real axis and 
$ \calS_{R}^d $ the $ d $-fold cartesian product of $ \calS_{R} $ with itself. 
Then we have the following proposition.
\begin{proposition}\label{prop:shallowholo}
Let $ R > 0 $, $\sigma\colon \calS_{R}^{k} \to \CC^{k}$ 
be holomorphic and $ 0<\eps < 2 R $. 
Assume given a sequence $ \bmbeta = (\beta_j)_{j\in \NN} \in \ell^{p}(\NN)$, 
with some $ 0< p < 1 $, $ b \in \RR^{k} $ 
and $ W = (W_{ij})_{\substack{1 \le i \le k \\1 \le j < \infty}} \in \RR^{k \times \NN} $. 

Then, if $ \max_{i= 1,\ldots,k}\abs{W_{ij}} \le \beta_j \ \forall j \in \NN $ there holds that
$ f_{W,b} \circ \xi\colon U \to \RR^{k}  $ is $ (\bmbeta,p,\eps) $-holomorphic on polytubes.
\end{proposition}
\begin{proof}
Let $ \bmrho $ be a $  (\bmbeta,\eps) $-admissible sequence, 
$ z=y + i \eta \in \calT_{\bmrho} $ and let 
$\tilde{z} := W \xi(z) + b$.  
Since $ W,b $ are real valued there holds
\begin{equation*}
\Im(\tilde{z}) = \frac{1}{2}W \eta.
\end{equation*}
Note that 
$\eta_j = \Im(z_j) < \rho_j-1$ for all $z\in \calT_{\bmrho}$; 
hence for all $ i = 1,\ldots, k$,
we obtain
\begin{align} \label{affineinstrip}
\abs{\Im(\tilde{z}_{i})} = \frac{1}{2}\abs{\sum_{j = 1}^{d} W_{ij} \eta_{j}} 
< 
\frac{1}{2}\sum_{j = 1}^{d} \beta_{j} (\rho_j - 1) \le \frac{\eps}{2} < R.
\end{align}
Therefore $ \tilde{z} \in \calS_{R}^d $ and $ f_{W,b}\circ \xi $ is holomorphic on $ \calT_{\bmrho} $, so that we proved the first condition of Definition \ref{def:bpepsholo}. 
To verify the second condition, 
let $ \tilde{\bmrho} $ be $ (\bmbeta,(\eps + 2R)/2) $-admissible and 
satisfy $ \rho_j < \tilde{\rho}_j $, and 
set $ \calO_{\bmrho} = \calT_{\tilde{\bmrho}} $ in Definition \ref{def:bpepsholo}.  
Hence $ \calT_{\tilde{\bmrho}} \supsetneq \closure{ \calT_{\bmrho}} $ 
and
\begin{equation*}
\sup_{z \in \calT_{\tilde{\bmrho}} }
\abs{ f_{W,b}(z)} = \max_{z \in \closure{ \calT_{\tilde{\bmrho}}} }  \abs{ f_{W,b}(z)} 
\le 
C_{\eps},
\end{equation*}
where we used that for all 
$ z \in \closure{ \calT_{\tilde{\bmrho}}}$, 
$\tilde{z} = Wz + b$ there holds 
$\abs{\Im(\tilde{z}_{i})} \le (\eps + R)/2 $ for all $ i = 1, \ldots, k$ 
and hence $ f_{W,b}(z) $ is bounded on this compact set. 
Note that $ C_{\eps} $ depends on the output dimension $ k $ but is independent on the input dimension $ d $.
\end{proof}
We now verify the assumptions of Proposition \ref{prop:shallowholo} 
for a selection of popular activation functions in the following examples.
\begin{example}
Let $ \sigma $ given by $ \sigma(y) := \left (\frac{1}{1 + e^{-y_1}}, \ldots, \frac{1}{1 + e^{-y_d}} \right ) $. 
The logistic function $ \frac{1}{1+e^{-z}} $, $ z \in \CC $ is meromorphic  
with poles at $ z = \pi i + 2ni\pi, n \in \ZZ$. 
In particular it is holomorphic on the open strip $\calS_{\pi}$. 
Thus, $ \sigma $ is holomorphic on $ \calS_{\pi}^{d} $ by Hartogs Theorem and 
uniformly bounded on any compact set contained in $ \calS_{\pi}^{d} $. 
\end{example}

\begin{example}
Let $ \sigma $ given by $ \sigma(y) := \left (\tanh(y_1), \ldots, \tanh(y_d)\right ) $. The  function $ \tanh(z) $, $ z \in \CC $ is meromorphic with poles at $ z = \frac{\pi}{2} i + n i \pi, n \in \ZZ $. In particular it is holomorphic on $ \calS_{\frac{\pi}{2}} $. Thus, $ \sigma $ is holomorphic on the strip $ \calS_{\frac{\pi}{2}}^{d} $ by Hartogs Theorem and uniformly bounded on any compact set contained in $ \calS_{\frac{\pi}{2}}^{d} $. 
\end{example}

\begin{example}
Let $ \sigma $ be the softmax function $ \sigma(y):= \left (\frac{e^{y_1}}{\sum_{j}e^{y_j}}, \ldots, \frac{e^{y_d}}{\sum_{j}e^{y_j}}\right ) $. Each component $ \frac{e^{z_k}}{ \sum_j e^{z_j}} $, $ z \in \CC^{d} $ is also meromorphic. 
In particular, we show that it is holomorphic on the strip $ \calS_{\frac{\pi}{2}}^{d} $. 
In fact, writing $ e^{z_{k}} = e^{y_k}(\cos(\eta_k) + i\sin(\eta_k)) $, with $ y_{k}, \eta_{k} \in \RR $ yields 
\begin{equation*}
\sum_{k= 1}^{d} e^{z_k} = 0 \iff 
\begin{cases}
\sum_{k= 1}^{d} e^{y_k} \cos(\eta_k) = 0 & \\
\sum_{k= 1}^{d} e^{y_k} \sin(\eta_k) = 0 &
\end{cases}
,
\end{equation*}
but $ \cos(\eta_k) > 0 $ for all $ k $ due to $ z_k \in \calS_{\frac{\pi}{2}} $. Thus, $ \sum_{k= 1}^{d} e^{y_k} \cos(\eta_k) > 0 $.
\end{example}

\begin{lemma}\label{lem:compholo}
Let $ X,W $ be Banach spaces, $ \bmbeta \in \ell^p(\NN) $ and $ \eps> 0  $. Let $ B_{X}(R) \subseteq X $ be the open ball on $ X $ centered at the origin with radius $ R $. 
Assume that $ \phi \colon U \to X $ is $ (\bmbeta,p,\eps) $-holomorphic with uniform bound $ C_{\eps} $ on a family of sets $  \set{\calO_{\bmrho} : \ \bmrho  \text{ is } (\bmbeta,\eps)\text{-admissible} } $ in the sense of Definition \ref{def:bpepsholo}, with $ \calO_{\bmrho} \supset \closure{\calT_{\bmrho}} $ and assume that $ h \colon B_{X}(C_{\eps} + \delta) \to W $ is holomorphic for some $ \delta > 0$.

Then, $ h\circ\phi \colon U \to W $ is $ (\bmbeta,p,\eps) $-holomorphic.
\end{lemma}
\begin{proof}
	Fix a $ \bmrho $ that is $ (\bmbeta,\eps) $-admissible, and a $ z \in \calO_{\bmrho} $ then	
	\begin{equation*}
		\norm{h\circ\phi(z)}{W} \le \sup_{\norm{\psi}{X} \le C_{\eps} } \norm{h(\psi)}{W} = \sup_{\psi  \in \closure{B_{X}(C_{\eps})} } \norm{h(\psi)}{W} \le \hat{C}_{\eps}
	\end{equation*}
	since $ h $ is bounded on this compact set. Thus
	\begin{equation*}
	\sup_{z \in \calO_{\bmrho}} \norm{h\circ\phi(z)}{W} \le \hat{C}_{\eps}
	\end{equation*}
	and $ \phi(\calT_{\bmrho}) \subseteq B_{X}(C_{\eps} + \delta ) $ 
implies that $ h \circ \phi $ is well defined on $ \calT_{\bmrho} $ and holomorphic as composition of holomorphic functions.
\end{proof}

\begin{proposition} \label{prop:generalizSPOD}
Let $ q \in 2{\mathbb N}$ be even,  
$ \bar{g}\circ \xi :U \to \CC^{d_{1}} $ be $ (\bmbeta,p,\eps) $-holomorphic for some $ 0 < p < 1 $ 
and 
assume that $f_{W^{(1)},b^{(1)}}$ satisfies 
the hypotheses of Proposition \ref{prop:shallowholo}.

 Then, 
for any shallow neural network $ \phi_{\theta} =  W^{(2)} f_{ W^{(1)},b^{(1)} } + b^{(2)}$ 
there exist $ C > 0, J \in \NN $ independent of $ d $  such that
\begin{equation*}
\norm{(g - \phi_{\theta})^q}{d,\alpha,\bmgamma,1,\infty} \le C
\end{equation*} 
for some SPOD weights $ \bmgamma $ defined by 
\begin{equation}\label{def:weightsSPOD}
\gamma_{\frku} := \sum_{\bm{\nu} \in \set{1:\alpha}^{\abs{\frku}}} \abs{\bm{\nu}}! \prod_{j \in \frku} \left ( 2^{\delta(\nu_j, \alpha )}  \tilde{\beta}_{j}^{\nu_j}\right )
\end{equation}
where $\tilde{\beta}_{j} = 2^{\alpha+2} \norm{\beta}{\ell^1(\NN)}/\eps$ 
$ \forall j \le J $ and $ 0 < \tilde{\beta}_{j} \le c \beta_{j} $ if $ j > J $. 
Moreover, the constant $ c > 0 $ is independent of $ d,\bm{\nu},\bmy $ and $ j $.
\end{proposition}

\begin{proof}
Since $ f_{W^{(1)},b^{(1)}}  \circ \xi $ is $ (\bmbeta,p,\eps) $-holomorphic 
by Proposition \ref{prop:shallowholo} 
and since affine transformations are holomorphic on the entire complex plane,
it is easily verified that 
$(\bar{g} - \phi_{\theta}) \circ \xi$ 
is  
$(\bmbeta,p,\eps)$-holomorphic for the same sequence $ \bmbeta $ and $ \eps $. 

Furthermore, the assumption that
$q$  is an even integer ensures that the 
map $h \colon \RR^{d_{1}} \to \RR $,  $x \mapsto \abs{x}^{q}$ 
admits holomorphic continuation on $ \CC^{d_{1}}$. 
Hence, Lemma \ref{lem:compholo} and \cite[Theorem 3.1, Remark 4.3]{DLS16} 
imply that there exists a finite set $ E:= \set{1,\ldots,J} \subset \NN $ such that

\begin{align*}
\norm{(g - \phi_{\theta})^q}{d,\alpha,\bmgamma,1,\infty} & \le C \sup_{\frku \in \set{1:d}} \frac{1}{\gamma_{\frku}} \sum_{\bm{\nu} \in \set{1:\alpha}^{\abs{\frku}}} \bm{\nu}_{\frku \cap E}! \prod_{j \in \frku \cap E} \left ( 2^{\delta(\nu_j, \alpha )} \tilde{\beta}_{j}^{\nu_j}\right ) \abs{\bm{\nu}_{\frku \cap E^c}}! \prod_{j \in \frku \cap E^c} \left ( 2^{\delta(\nu_j, \alpha )} \tilde{\beta}_{j}^{\nu_j}\right ) \\
& \le C \sup_{\frku \in \set{1:d}} \frac{1}{\gamma_{\frku}} \sum_{\bm{\nu} \in \set{1:\alpha}^{\abs{\frku}}} \abs{\bm{\nu}}! \prod_{j \in \frku} \left ( 2^{\delta(\nu_j, \alpha )} \tilde{\beta}_{j}^{\nu_j}\right )
\end{align*} 
for a $ C $ independent of the dimension $ d $ but dependent on $ \norm{\bmbeta}{\ell^1(\NN)} $ and $ \eps $, which proves the claim upon choosing the weights $ \bmgamma $ as in \eqref{def:weightsSPOD}.
\end{proof}
\subsection{Quantified Holomorphy of Deep Neural Networks}
\label{sec:HolomDNN}
In this section we prove 
$ (\bmbeta,p, \eps) $-holomorphy for a class of DNNs, with $ L \ge 3 $ in \eqref{def:dnn},
to which Proposition \ref{prop:generalizSPOD} can be extended.
We verify $ (\bmbeta,p, \eps) $-holomorphy provided that 
a summability condition on the network weights holds.
\begin{proposition} \label{prop:deepholo}
Let $ R,R' > 0 $, 
$\sigma \colon \calS_{R}^{d_{\ell}} \to \calS_{R'}^{d_{\ell}}$ be holomorphic 
and let $ 0 < \eps < 2R$ be given.
Assume given a 
sequence 
$ \bmbeta = (\beta_{j})_{j \in \NN} \in \ell^p(\NN)$, 
with some $ 0 < p < 1 $, 
and with the network parameters 
$b^{(\ell)} \in \RR^{d_{\ell}} \ \forall \ell \ge 1$,  
$W^{(1)} \in \RR^{d_1 \times \NN}$ 
and 
$ W^{(\ell)} \in \RR^{d_{\ell} \times d_{\ell-1}} $ for $ \ell \ge 2 $. 
Assume that $ \max_{i= 1,\ldots,d_{1}} \abs{W_{ij}^{(1)}} \le \beta_{j} \  \forall j \in \NN $ 
and that for all $ \ell = 2,\ldots,L-1 $ there holds
\begin{equation}\label{normW}
\max_{i= 1,\ldots,d_{\ell}} \sum_{j = 1}^{d_{\ell-1}} \abs{W_{ij}^{(\ell)}} \le \frac{R}{R'}.
\end{equation} 
 
Then, 
$ \phi_{\theta}^L \circ \xi $ is $ (\bmbeta,p,\eps) $-holomorphic on polytubes.
\end{proposition}
\begin{proof}
By Proposition \ref{prop:shallowholo} $ f_{W^{(1)},b^{(1)}} \circ \xi $ is $ (\bmbeta,p,\eps) $-holomorphic. Moreover, \eqref{affineinstrip} and the hypothesis that $ \sigma \colon \calS_{R}^{d_{\ell}} \to \calS_{R'}^{d_{\ell}} $ give that the image of $ \calT_{\bmrho} $ under the map $ f_{W^{(1)},b^{(1)}} \circ \xi $ is contained in $ \calS_{R'}^{d_{1}} $ for all $ (\bmbeta,\eps) $-admissible sequences $ \bmrho $. 
Next, for fixed $ \ell \in \set{2,\ldots,L-1} $, we define $ \tilde{z} = W^{(\ell)}z + b^{(\ell)} $. Then, since $ W^{(\ell)},b^{(\ell)} $ are real valued, \eqref{normW} gives
\begin{align*}
\max_{i= 1,\ldots,d_{\ell}}\abs{\Im(\tilde{z_i})} & 
= \max_{i= 1,\ldots,d_{\ell}}\abs{\sum_{j = 1}^{d_{\ell-1}}W_{ij}^{(\ell)}\Im( z_j )}
\le \frac{R}{R'} \max_{j= 1,\ldots,d_{\ell-1}} \abs{\Im (z_j)},
\end{align*}
which implies that the affine transformations for $ \ell= 2,\ldots,L-1$ 
map $ \calS_{R'}^{d_{\ell-1}} $ to $ \calS_{R}^{d_{\ell}}$. 
Hence, the network map $ \phi_{\theta}^L \circ \xi $ is well defined and holomorphic 
as composition of holomorphic functions on such $ \calT_{\bmrho}$. 

The second condition of $(\bmbeta,p,\eps)$-holomorphy 
is verified analogously to Proposition \ref{prop:shallowholo}: 
for a polyradius $ \tilde{\bmrho} $ that is 
$(\bmbeta,(\eps + 2R)/2)$-admissible and satisfies 
$\rho_j < \tilde{\rho}_j $ there holds 
$\closure{ \calT_{\bmrho}}\subseteq  \calT_{\tilde{\bmrho}}$ 
and $ \phi_{\theta}^{L} \circ \xi $ is continuous, 
hence bounded on $ \closure{\calT_{\tilde{\bmrho}}}$ 
by some constant $ C_{\eps} $ independent of $ \bmrho $.
\end{proof}

\begin{remark}\label{remark:archdep}
It is of interest to track the dependence of the uniform bound $ C_{\eps} $ on the architecture. The hypotheses of Proposition \ref{prop:deepholo} set a constraint on the width of the inner layers of the network. Very wide networks are still possible, but at the price of smaller weights allowed in the analysis. On the other hand, no formal limitation is imposed on the biases $  b^{(\ell)}, \ell = 1,\ldots,L $ and on the output weight matrix $ W^{(L)} $, although these parameters also contribute to the bound $ C_{\eps} $. Furthermore, $ C_{\eps} $ is completely independent of the number of layers $ L $ provided that \eqref{normW} holds for all the intermediate layers.
\end{remark}
The abstract assumptions on the activation function $ \sigma $ 
in Proposition \ref{prop:deepholo} are next verified for a 
selection of widely used activation functions.
\begin{example}
Let $ \sigma $ be the standard logistic function. Then for $ z = y + i \eta \in \CC $ there holds
\begin{align*}
\Im(\sigma(z)) &= \Im\left (\frac{1}{1 + e^{- z} }\right ) 
=  \Im\left (\frac{1 + e^{-y + i \eta}}{(1 + e^{-y - i \eta} ) (1 + e^{-y + i \eta}) }\right ) \\
& = \Im\left (\frac{1 + e^{-y + i \eta}}{1 + e^{ - 2y} + 2e^{-y}\cos(\eta) }\right ) = \frac{1}{1 + e^{-2y} + 2e^{-y}\cos(\eta)} \Im (1 + e^{-y + i \eta}) \\
& = \frac{ e^{-y} \sin(\eta) }{1 + e^{ - 2y} + 2e^{- y}\cos(\eta)} .
\end{align*}
Thus if $ \abs{\eta} < \frac{\pi}{2} $, this gives $ \abs{\Im(\sigma(z))} < \frac{1}{2} $. Hence, $ \sigma \colon \calS_{R}^{d_{\ell}} \to \calS_{R'}^{d_{\ell}} $ is holomorphic for $ R = \frac{\pi}{2} $, $ R' = \frac{1}{2} $ and Proposition \ref{prop:deepholo} applies.
\end{example}

\begin{example}
Let $ \sigma $ be the hyperbolic tangent function. 
Then for $ z = y + i \eta \in \CC $  there holds
\begin{align*}
\Im(\sigma(z)) &= \Im\left (\frac{1 - e^{ - 2z}}{1 + e^{ - 2z} }\right ) 
=  \Im\left (\frac{(1 - e^{-2(y + i \eta)})(1 + e^{-2(y - i \eta)})}{(1 + e^{-2(y + i \eta)})(1 + e^{-2(y - i \eta)}) }\right ) \\
& = \Im\left (\frac{1 - e^{-4y} + 2 i e^{-2y}\sin(2 \eta)}{1 + e^{ - 4y} + 2e^{-2y}\cos(2 \eta) }\right ) = \frac{2 e^{-2y}\sin(2 \eta)}{1 + e^{ - 4y} + 2e^{-2y}\cos(2 \eta)}.
\end{align*}
Thus if $ \abs{\eta} < \frac{\pi}{4} $, this gives $ \abs{\Im(\sigma(z))} < 1 $. Hence, $ \sigma \colon \calS_{R}^{d_{\ell}} \to \calS_{R'}^{d_{\ell}} $ is holomorphic for $ R = \frac{\pi}{4} $, $ R' = 1 $ and Proposition \ref{prop:deepholo} applies.
\end{example}

\begin{example}
Let $ \sigma $ be the softmax activation function. 
Then, for $ z = y + i \eta \in \CC^{d_{\ell}} $ and for all components $ k = 1,\ldots,d_{\ell}$i,
there holds 
\begin{align*}
\Im(\sigma(z_{k})) &= \Im \left (\frac{e^{z_{k}} }{\sum_j e^{z_j}}\right ) = \frac{1 }{\abs{\sum_j e^{z_j}}^2} \Im \left (e^{z_{k}} \left (\sum_{j} e^{y_j} \cos(\eta_{j}) - i \sum_{j} e^{y_j} \sin(\eta_{j}) \right )\right ) \\
& = \frac{ e^{y_{k}} \sin(\eta_{k}) \sum_{j} e^{y_{j}}\cos(\eta_{j}) - e^{y_{k}} \cos(\eta_{k}) \sum_{j} e^{y_{j}}\sin(\eta_{j}) }{\left (\sum_j e^{y_j} \cos(\eta_{j})\right )^2 + \left (\sum_j e^{y_j} \sin(\eta_{j})\right )^2}.
\end{align*}
Thus, if $ \abs{\eta} < \frac{\pi}{4} $, this gives
\begin{equation*}
\abs{\Im(\sigma(z_{k}))} < \frac{\sqrt{2} e^{y_{k}}}{\frac{1}{2}\sum_j e^{y_j}} \le 2 \sqrt{2}.
\end{equation*}
Hence, $ \sigma \colon \calS_{R}^{d_{\ell}} \to \calS_{R'}^{d_{\ell}} $ is holomorphic for $ R = \frac{\pi}{4} $, $ R' = 2 \sqrt{2} $ and Proposition \ref{prop:deepholo} applies.
\end{example}

\begin{remark}
To simplify the notation, we restricted the present discussion to the case when all the activation functions are of 
the same type. 
However, the proof of Proposition \ref{prop:deepholo} remains valid verbatim 
if we use different activations $ \sigma^{(\ell)} \colon \calS_{R_{\ell}}^{d_{\ell}} \to \calS_{R'_{\ell}}^{d_{\ell}}$ 
in each layer. 
Specifically, we must in this case pick $ \eps < 2R_1 $ and replace \eqref{normW} by 
\begin{equation*}
\max_{i= 1,\ldots,d_{\ell}} \sum_{j = 1}^{d_{\ell-1}} \abs{W_{ij}^{(\ell)}} \le \frac{R_{\ell}}{R'_{\ell-1}}.
\end{equation*}
\end{remark}
\subsection{Estimate on the generalization gap}
\label{sec:EstGenGap}
In order to estimate the generalization gap, we need a further 
\emph{technical} assumption on the neural networks and the underlying map i.e. 
we assume that, for any fixed architecture $ L,d_1, \ldots,d_L $ 
the input function $ \bar{g} $ satisfies,

\begin{equation} \label{ass:approxerr}
0 < c(\bar{g},(d_1,\ldots,d_L)) := \inf_{d \ge 1}\inf_{\theta} \calE_{G}(\theta)
\end{equation}
where the second infimum is taken with respect to all network parameters 
$ \theta $ that satisfy the hypothesis of Proposition \ref{prop:deepholo}.

We observe that this constant is related to the best DNN approximation with fixed architecture, 
see Remark \ref{remark:cg} for more details.
There holds the following theorem on the generalization gap.
\begin{thm}\label{thm:qmcgap}
Let $ \alpha \in \NN_{\ge 2} $. 
Assume that the activation function $ \sigma \colon \calS_{R}^{d_{\ell}} \to \calS_{R'}^{d_{\ell}}$ 
and the set of parameters $ \theta $ satisfy the hypotheses of Proposition \ref{prop:deepholo}. 
Let further $ \bar{g}\circ \xi:U \to \CC^{d_{L}}$ 
be $ (\bmbeta,p,\eps) $-holomorphic for some $ p < \frac{1}{\alpha}$ 
and 
with $0<\eps<2R$ 
and 
assume that \eqref{ass:approxerr} holds. 

Then there exists a constant $C>0$ that is independent of $d$ and $N$ such that
\begin{equation*}
\abs{\calE_{G}(\theta) - \calE_{T}^{(\alpha)}(\theta)} \le C N^{-\alpha} \;.
\end{equation*}
\end{thm}
\begin{proof}
Applying the arguments in Proposition \ref{prop:generalizSPOD},
we obtain that 
$ (\bar{g} - \phi_{\theta}^{L})^2\circ \xi$ is $ (\bmbeta,p,\eps) $-holomorphic. 
Hence $  (g - \phi_{\theta}^{L})^2 $ belongs to the QMC weighted space 
$ \calW_{d,\alpha,\bmgamma, 1,\infty}$ and there holds the QMC approximation error convergence from \cite[Theorem 2.5]{DLS20_902} 
which gives
\begin{equation}
 \abs{\calE_{G}^2(\theta) - (\calE_{T}^{(\alpha)})^2(\theta)} \le C(\theta) N^{-\alpha} 
\end{equation}
for some constant $ C(\theta) > 0 $ that is independent of $ d $ and $ N $. 
Thus, by \eqref{ass:approxerr}, 
we obtain
\begin{align*}
\abs{\calE_{G}(\theta) - \calE_{T}^{(\alpha)}(\theta)} 
&= \frac{\abs{\calE_{G}^2(\theta) - (\calE_{T}^{(\alpha)})^2(\theta)}}{\calE_{G}(\theta) + \calE_{T}^{(\alpha)}(\theta)} 
\\
& \le \frac{1}{\inf_{\theta} \calE_{G}(\theta)} \abs{\calE_{G}^2(\theta) - (\calE_{T}^{(\alpha)})^2(\theta)}  
\le \frac{C(\theta)}{c(\bar{g},(d_1,\ldots,d_L))} N^{-\alpha}.
\end{align*}
Here, the constant $ c(\bar{g},(d_1,\ldots,d_L))$ is 
bounded away from zero as in \eqref{ass:approxerr}
uniformly with respect to $ N,d$, (but not with respect to $ (d_1,\ldots,d_L) $, cf. Remark \ref{remark:cg})
and the proof is complete. 
\end{proof}

\begin{remark} \label{remark:cg}
Condition \eqref{ass:approxerr} is not in contradiction with the universal approximation theorem 
(e.g. \cite{PinkusActa}) since we fix a priori the width of the inner layers of the networks. 
Note that \eqref{ass:approxerr} is true whenever $ g $ 
cannot be expressed as a DNN exactly a.e. in any parameter space $ [0,1]^d $ and, 
for fixed $ L,d_1,\ldots,d_L $, 
there holds that $ \inf_{\theta} \calE_{G}(\theta) $ is monotonically increasing with respect to $ d $. 
This is a reasonable assumption since high-dimensional input $ d $ reduces the approximation power of a DNN, 
when the rest of the architecture is fixed. 
In particular we expect that in most cases, 
for any parameter vector $ \theta $, the constant $ \frac{1}{\calE_{G}(\theta)} $ 
decreases when increasing $d$. 
On the other hand, 
in the case that \eqref{ass:approxerr} does not hold, 
for every $ \delta $ there exists $ \theta^{*} $ such that the approximation error 
$ \calE_{A} := \norm{g -  \phi_{\theta^{*}}^L}{L^{\infty}([0,1]^d)} < \delta $  
and hence there holds 
$\calE_{G}, \calE_{T}^{(\alpha)} < \delta $
for the parameter sequence $ \theta^{*} $.
\end{remark}

Theorem \ref{thm:qmcgap} implies the decay of $ \calE_{G} $ as
\begin{equation}\label{decayEG}
\calE_{G}(\theta) \le \frac{C(\theta)}{c(\bar{g},(d_1,\ldots,d_L))} N^{-\alpha} + \calE_{T}^{(\alpha)}(\theta),
\end{equation}
for all $ \theta $ satisfying the decay assumptions of Proposition \ref{prop:deepholo}. 
One natural question is what happens when the network size grows and $ c(\bar{g},(d_1,\ldots,d_L)) $ becomes smaller as a consequence. We note that in this case we have the direct, but coarse, estimate 
\begin{equation*}
\abs{\calE_{G}(\theta) - \calE_{T}^{(\alpha)}(\theta)} \le \sqrt{\abs{\calE_{G}^2(\theta) - (\calE_{T}^{(\alpha)})^2(\theta)}} \le C(\theta)^{1/2} N^{-\alpha/2}.
\end{equation*}
Moreover, we have that $ C(\theta) $ is proportional to (the square of) the constant $ C_{\eps} $ of Definition \ref{def:bpepsholo} (see \cite[Theorem 3.1 and Proposition 4.1]{DLS16}) and thus is independent on the architecture of $ \phi_{\theta} $ by Remark \ref{remark:archdep}. 
This implies a guaranteed order of $ N^{-\alpha/2} $ 
independent of the network architecture whenever Proposition \ref{prop:deepholo} applies. 
However, we observe in the numerical experiment that rate $ \alpha $ is attained in many situations.
%
\subsection{Impact of Discretization on DNN surrogate fidelity}
\label{sec:ImpDiscr}
The key result, Theorem \ref{thm:qmcgap}, provided an apriori 
bound on the generalization gap between the DNN
and the DtO map of interest, $g$. In practice, and in the 
numerical experiments reported in Section \ref{sec:4} ahead, 
however, the DNN training algorithms will, in general, not
be able to numerically access the exact map $g$.
Instead, 
a numerical approximation $g_\delta$ is available, 
which stems from some discretization scheme applied  
to a governing PDE, with 
discretization parameter $\delta \in (0,1]$. 
The DNN training process will numerically access $g_\delta$.
Completion of DNN training, therefore, 
will result in a DNN $\phi_{\theta^{*}}^{L} = \hat{g}_\delta$ instead of the DNN $\hat{g}$.
 With $|Y|\leq 1$,
\begin{equation} \label{errsplit}
\left (\int_Y |g(y)-\hat{g}_\delta(y)|^2 \dd y \right  )^{1/2} 
\le 
\sup_{y \in Y} |g(y) - g_\delta(y)| + \left (\int_Y  |g_\delta(y)-\hat{g}_\delta(y)|^2 \dd y \right  )^{1/2}
\end{equation}
Assuming $\{ g_\delta \}_{0<\delta\leq 1}$ to be convergent,
i.e. \eqref{eq:gdeltcnsist} holds, it remains to control the \emph{generalization error subject to discretized DtO maps},
$|g_\delta(y)-\hat{g}_\delta(y)|$. The proposed DNN training algorithm
requires ensuring applicability of Theorem \ref{thm:qmcgap} to 
the discretized DtO map $g_\delta$ rather than to the truth $g$.
To this end, we require \emph{uniform (w.r. to $\delta \in (0,1]$) parametric
$ (\bmbeta,p,\eps) $-holomorphy of $y\mapsto g_\delta(\xi(y))$ 
for some $ p < \frac{1}{\alpha}$}, 
i.e. where the conditions in \eqref{bpepseq} are met with $ \bmbeta,\eps,C_{\eps} $ independent of $ \delta $.
Generally, 
for DtO maps $\bar{g}\circ \xi$ which are $ (\bmbeta,p,\eps) $-holomorphic,
and for $g_\delta$ obtained by 
\emph{discretizations which are uniformly} (w.r.\@ to parametric inputs in 
a complex neighborhood of input data $Y$) \emph{stable}, 
uniform 
$ (\bmbeta,p,\eps) $-holomorphy of the family $\{g_\delta\}_{0<\delta\leq 1}$
follows from that of the DtO map $g$.
 As a consequence, this stability allows to bound the generalization error in the second term of \eqref{errsplit}, as we did in \eqref{decayEG} for the true map $ g $.
This uniform discretization stability required to apply Theorem \ref{thm:qmcgap} 
must be verified on a case-by-case basis.

For the benefit of the reader, we shall verify uniform 
$(\bmbeta,p,\eps)$-holomorphy of FEM discretizations of 
a model PDE example in Section \ref{sec:unifgalbpeps}.
\section{Numerical Experiments}
\label{sec:4}
\subsection{Implementation}
\label{sec:Implem}
The implementation and testing of the DL-HoQMC algorithm is performed using 
the machine learning framework PyTorch \cite{torch}. 
Both the training and test sets (to approximate the generalization error \eqref{eq:gerr} 
with QMC quadrature) are based on points generated with previously described 
extrapolated (and interlaced) lattice rules, where the size of the test set 
is chosen such that it significantly outnumbers the size of the training set, i.e. 
at least twice the amount of testing points than for the biggest training set. 
The training is performed using the ADAM optimizer \cite{adam} in a full-batch mode 
using a maximum of 20k epochs (learning steps). 
Hyperparameters, listed in \Tref{tab:ensemble_params} are selected via an ensemble training process, 
as described in \cite{LMR1}. 
The weights of the networks are initialized based on the so-called Xavier normal initializer \cite{xavier}, 
which is standard for training networks with sigmoid (or tanh) activation functions. 
\begin{table}[htbp]
  \caption{Hyperparameters for the ensemble training}
  \label{tab:ensemble_params}
  \centering
  \begin{tabular}{ll}
    \toprule
    \cmidrule(r){1-2}
    Hyperparameter     & values used for the ensemble \\
    \midrule
learning rate & $10^{-4}$\\
regularization parameter $\lambda$ &  $10^{-5},10^{-6},10^{-7} $\\
depth $L$ & $2^{2},2^{3},2^4$\\
width $d_j$ (constant) & $3\times2^1,3\times2^2,3\times2^3$\\
\# initializations & $2$  \\
    \bottomrule
  \end{tabular}
\end{table}
We train the networks for EPL points based on the upper bound \eqref{eq:epllf} and plot the training error \eqref{eq:epllf1}. For IPL points, however, we use the standard mean-square error \eqref{eq:ipllf} for training and plot the training error based on the $L^2$ norm \eqref{eq:plf}. Furthermore, we emphasize that all results are based on quantities (i.e. $\mathcal{E}_G$, $\mathcal{E}_T$ and so on) which are averaged over the trained ensemble of networks. This implies a certain stability towards the choice of the DNN architecture in order to obtain the desired rate of convergence.
If not stated differently, we base the subsequent experiments on EPL training points. 
Rates of convergence for all subsequent experiments are estimated using the exponential 
of the least squares fit of the logarithmized data, where the first order coefficient 
of the least squares fit defines the rate of decay. The scripts to perform the ensemble training together with all data sets used in the experiments can be downloaded from \href{https://github.com/tk-rusch/DL-HoQMC}{\textit{https://github.com/tk-rusch/DL-HoQMC}}.

\subsection{Function approximation}
\label{sec:FncAppr}
For our first numerical experiment, we approximate the following function:
\begin{equation}
\label{eq:weighted_sum}
g(y) = \frac{1}{1 + 0.5 \sum_{j=1}^d y_j j^{-2.5}},
\end{equation} 
with the DL-HoQMC algorithm.  Here, there is no additional discretization error as mentioned in Section \ref{sec:ImpDiscr}.
In \Fref{fig:func_approx_50d}, we present the generalization error of the trained neural networks for different number of EPL training points together with the training error for $d=50$, i.e. a $50$-dimensional parameter space. 
We can see that while the training error is very low and does not seem to decay 
with increasing number of training points, 
the generalization error decays with a rate of around $2.1$ in terms of $\#(\train)$
and thus the generalization gap should decay at the same rate. 
This is indeed verified in \Fref{fig:fun_appr_50d_gap}, 
where we observe a decay rate of approximately $2.3$ for the generalization gap, 
which agrees with our theoretical predictions. 
Furthermore, we also train DNNs based on the interlaced lattice rule (IPL). 
\Fref{fig:fun_appr_50d_gap} also shows the generalization gap using IPL points.
We can see that the generalization gap based on the IPL rule 
has approximately the same convergence rate as using the EPL rule, 
i.e. a convergence rate of around $2.3$. Note that in this experiment a bigger learning rate is used, i.e. a learning rate of $10^{-3}$.

\begin{figure}[h!]
\centering
\begin{minipage}[t]{.49\textwidth}
\includegraphics[width=1.\textwidth]{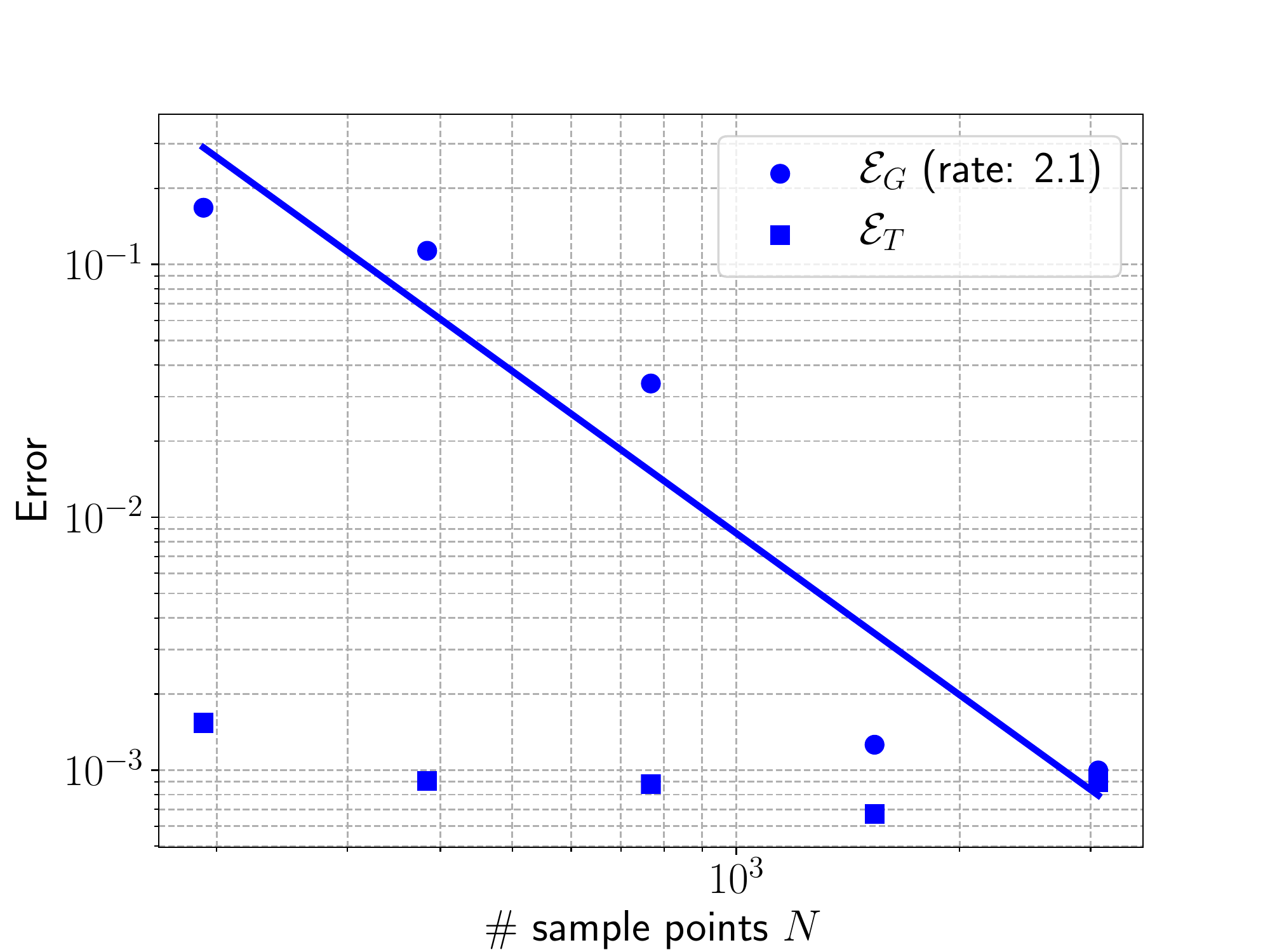}
\caption{Generalization error $\mathcal{E}_G$ for approximating the function $g$ \eqref{eq:weighted_sum} in $d=50$ dimensions together with the training error $\mathcal{E}_T$.}
\label{fig:func_approx_50d}
\end{minipage}%
\hspace{0.01\textwidth}
\begin{minipage}[t]{0.49\textwidth}	
\includegraphics[width=1.\textwidth]{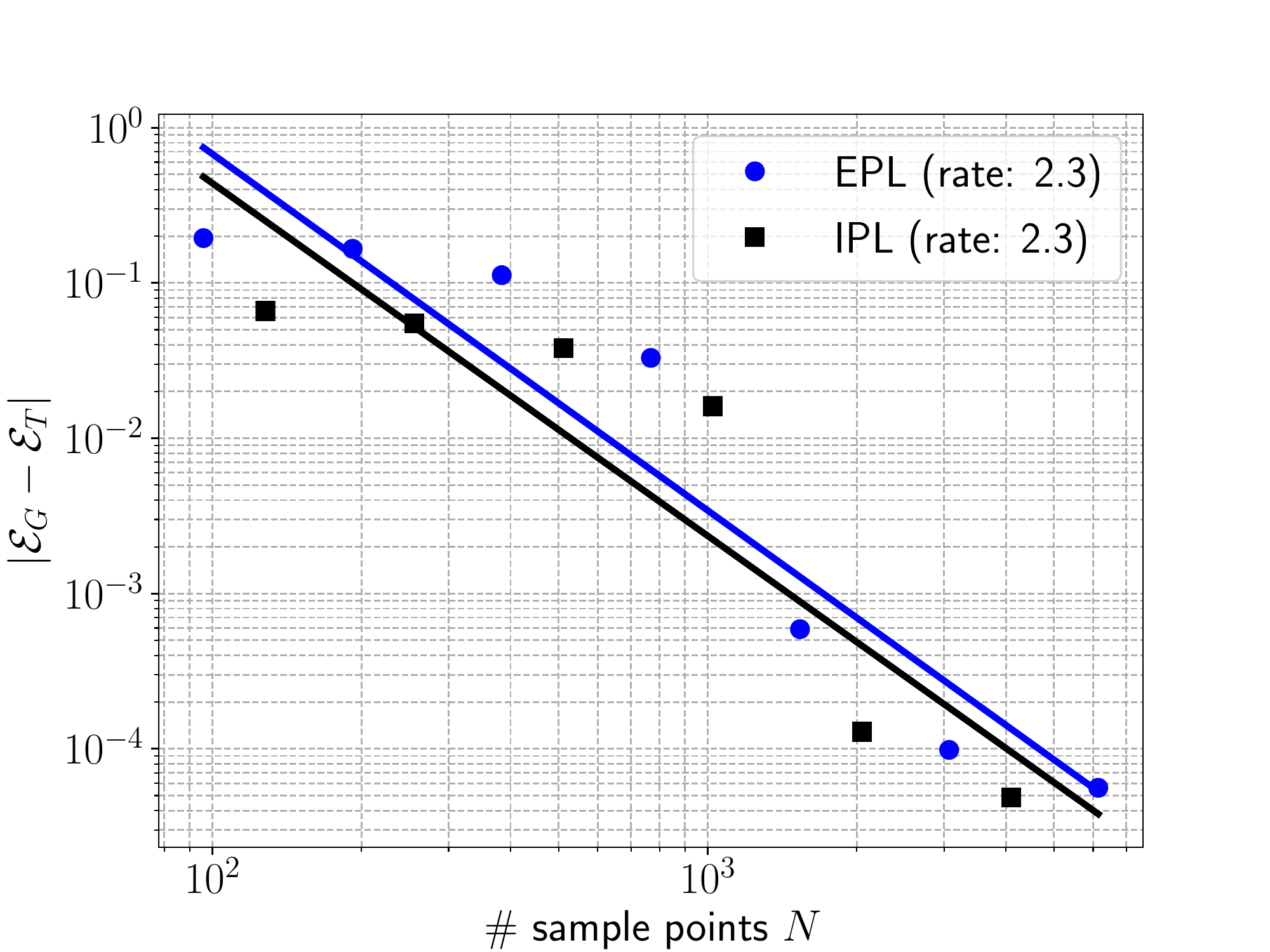}
\caption{
Generalization gap for approximating the function $g$ \eqref{eq:weighted_sum} 
with parameter dimension $d=50$ using both methods, 
the extrapolated polynomial lattice rule (EPL) as well as the interlaced polynomial lattice rule (IPL).}
\label{fig:fun_appr_50d_gap}
\end{minipage}
\end{figure}

We also observe that the IPL and EPL based designs $\train$ are considerably more economic
than deterministic, tensor product constructions: two points per coordinate imply then
$\#(\train) = 2^d$, i.e. $2^{50} \simeq 10^{16}$ points. 
This is to be compared to $10^3 - 10^4$ points in the lattices used in our numerical examples.
Using $\train$ based on i.i.d random draws, a reduction of the (mean square) generalization gap by
a factor of $10^4$ as displayed in \Fref{fig:fun_appr_50d_gap} would mandate a $10^{8}$ fold
increase of $\#(\train)$, i.e. about $\#(\train) \sim 10^{10}$ which is likewise prohibitive.

\subsection{Elliptic parametric PDEs}
\label{subsec:elliptic_exp}
For the second numerical experiment, 
we consider a test case that arises as a prototype for uncertainty quantification (UQ) 
in elliptic PDEs with uncertain coefficient \cite{CCS15,CDS10,DLS20_902}.
On the bounded physical domain $ D = (0,1)^2 $ and with parameters
$ y \in \left [ - \frac{1}{2}, \frac{1}{2} \right ]^{d}$, 
we consider the following elliptic equation with homogeneous Dirichlet boundary conditions
\begin{equation}\label{ellPDE}
	\begin{cases}
	-\div \left (a(x,y) \nabla u(x,y) \right ) = f(x) & x \in D 
        \\
	u (x,y)= 0 & x \in \partial D
	\end{cases}
	.
\end{equation}
We choose a deterministic source $ f(x) = 10 x_1 $, and the following observable 
\begin{equation*}
g(y) := \frac{1}{|\tilde{D}|}\int_{\tilde{D}} u(\cdot,y),
\end{equation*}
for $ \tilde{D} = (0,\frac{1}{2})^2 $. We denote the approximated observable, i.e.\ the observable based on a FEM approximation $u_h$ of $u$, as $g_h$.

We assume that the diffusion coefficient $ a(x,y) $ is affine-parametric, that is
\begin{equation} \label{affinePar}
	a(x,y) = \bar{a}(x) + \sum_{j = 1}^{d} y_j \psi_j(x).
\end{equation}
Here we choose $ \bar{a} \equiv 1 $, $ \psi_{j}(x) := \psi_{(k_1,k_2)}(x) = \frac{1}{(k_1^2 + k_2^2)^{\eta}}\sin(k_1\pi x_1)\sin(k_2\pi x_2)$ 
and the ordering is defined by 
$(k_1,k_2) < (\bar{k}_1,\bar{k}_2) $ when $ k_1^2 + k_2^2 < \bar{k}_1^2 + \bar{k}_2^2$ 
and is arbitrary when equality holds.  In particular this ensures well-posedness of the problem and the asymptotic decay 
\begin{equation} \label{betajell}
 	\beta_{j} := \frac{\norm{\psi_{j}}{L^{\infty}(D)}}{2 \operatorname{essinf}_{x \in D}\bar{a}(x)}  \sim  j^{-\eta}  
\end{equation} 

which in turn implies that $ u $, (hence $ g $) is $ (\bmbeta,p,\eps) $-holomorphic for any $ p > \frac{1}{\eta} $, 
with the arguments from \cite{CCS15}. 
\subsubsection{\texorpdfstring{Uniform $(\bmbeta,p,\eps)$-holomorphy}{}} 
\label{sec:unifgalbpeps}
We start by sketching the basic arguments to show 
uniform $ (\bmbeta,p,\eps) $-holomorphy in the sense of Section \ref{sec:ImpDiscr}, 
for the case of linear, second order elliptic PDEs in the divergence form \eqref{ellPDE} \eqref{affinePar}. 
First we rewrite the weak formulation of \eqref{ellPDE} for $ z \in \CC^{\NN} $: 
denoting $ V := H_0^1(D;\CC) $ the complexified Banach space of $ H_0^1 $, 
we seek 
\begin{equation} \label{weakellpde}
u\in V \;\;\mbox{such that}\;\;
\frka_{z} (u,v) =\bracket{ f, v} \quad \forall v \in V 
\;.
\end{equation} 
Here the complex extension of the parametric bilinear form is the 
sesquilinear form $ \frka_{z} \colon V \times V \to \CC $ defined by
\begin{equation*}
\frka_{z} (w,v) := \int_D a(\cdot,z) \nabla w \cdot \overline{\nabla v}
\end{equation*}
and the brackets denote the pairing of $ V $ with its (topological) dual $ V^* $ over the field $ \CC $. 
Now, let $ \bmrho $ be $ (\bmbeta, \eps) $-admissible for the 
same (real) $ \bmbeta $ as in \eqref{betajell} and $ 0<\eps < (1 - \norm{\bmbeta}{\ell^1(\NN)})/2 $. 
Then there holds  uniform coercivity of $ \frka_{z} $ on $ V $ as  $ \forall z\in \calT_{\bmrho} $,
\begin{equation}\label{essinfa}
\abs{a(x,z)} \ge \Re ( a(x,z) ) 
\ge 
\bar{a}(x)\left (1 - \sum_{j \ge 1} \Re (z_j)  \beta_{j} \right ) 
\ge 
(1 - \eps - \norm{\bmbeta}{\ell^1(\NN)}) \operatorname{essinf}\bar{a} > 0.
\end{equation}
 Uniform (on $\calT_{\bmrho}$)
continuity is readily verified, 
so that the complexified variational form of the PDE \eqref{ellPDE}
is well-defined by the (complex version of the) 
Lax-Milgram lemma. 
Furthermore,
\eqref{essinfa} directly implies also uniform coercivity for any conforming 
FE discretization space $ V_{h} \subset V$, 
including in particular the first order, 
continuous Lagrangian FEM with mesh width $ 0 < h \le 1 $,
used in our numerical experiments. 
Therefore, the discrete problem obtained replacing 
$ V $ with $ V_{h} $ in \eqref{weakellpde} admits a 
unique solution $ u_{h} \in V_{h} $ and 
there holds the uniform (with respect to $z\in \calT_{\tilde{\bmrho}}$ and $ h $) 
bound  
\begin{equation*}
\sup_{0<h\le 1}\sup_{z \in \calT_{\tilde{\bmrho}}} \abs{g_{h}(z)} 
\lesssim 
\sup_{0<h\le 1}\sup_{z \in \calT_{\tilde{\bmrho}}} 
\norm{u_{h}(\cdot,z) }{V} 
\le \frac{\norm{f}{V^*}}{ (1 - 2\eps - \norm{\bmbeta}{\ell^1(\NN)} ) \operatorname{essinf}\bar{a} } 
=:C_{\eps}
\end{equation*}
for some $ \tilde{\bmrho} > \bmrho $.

\subsubsection{Results with EPL training points}
Given the preceding section, it is clear that we can use a FEM simulation in order to provide training data for the deep neural networks approximating the underlying observable. For the first example, we set $ \alpha = 2, \eta = 2.5,$ and we run first order Lagrangian finite element (FEM) simulations with
$131072$ triangular elements for each QMC sample point, in order to generate the training data.

\par
The resulting generalization gap (averaged over the ensemble) for approximating the observable $g_h$ on parameter domains of dimension $16$ and $32$,
with EPL training points, can be seen in \Fref{fig:elliptic_pde_dim_comp}. 
We observe that for both choices of the parameter dimension $d$,
the generalization gap not only has the same convergence rate of approximately $2.1$, 
but also has almost the same absolute error. 
This is again in accordance with our theoretical findings, 
where we claimed the generalization gap to be independent of the dimension of the underlying problem. Note that we slightly increased the depth of the networks for the $32$ dimensional case in order to have similar training error, i.e. instead of using networks with depths of $4,8$ and $16$, we use depths of $6,10$ and $18$ for the ensemble training procedure.
\begin{figure}[h!]
\centering
\begin{minipage}[t]{0.49\textwidth}	
\includegraphics[width=1.\textwidth]{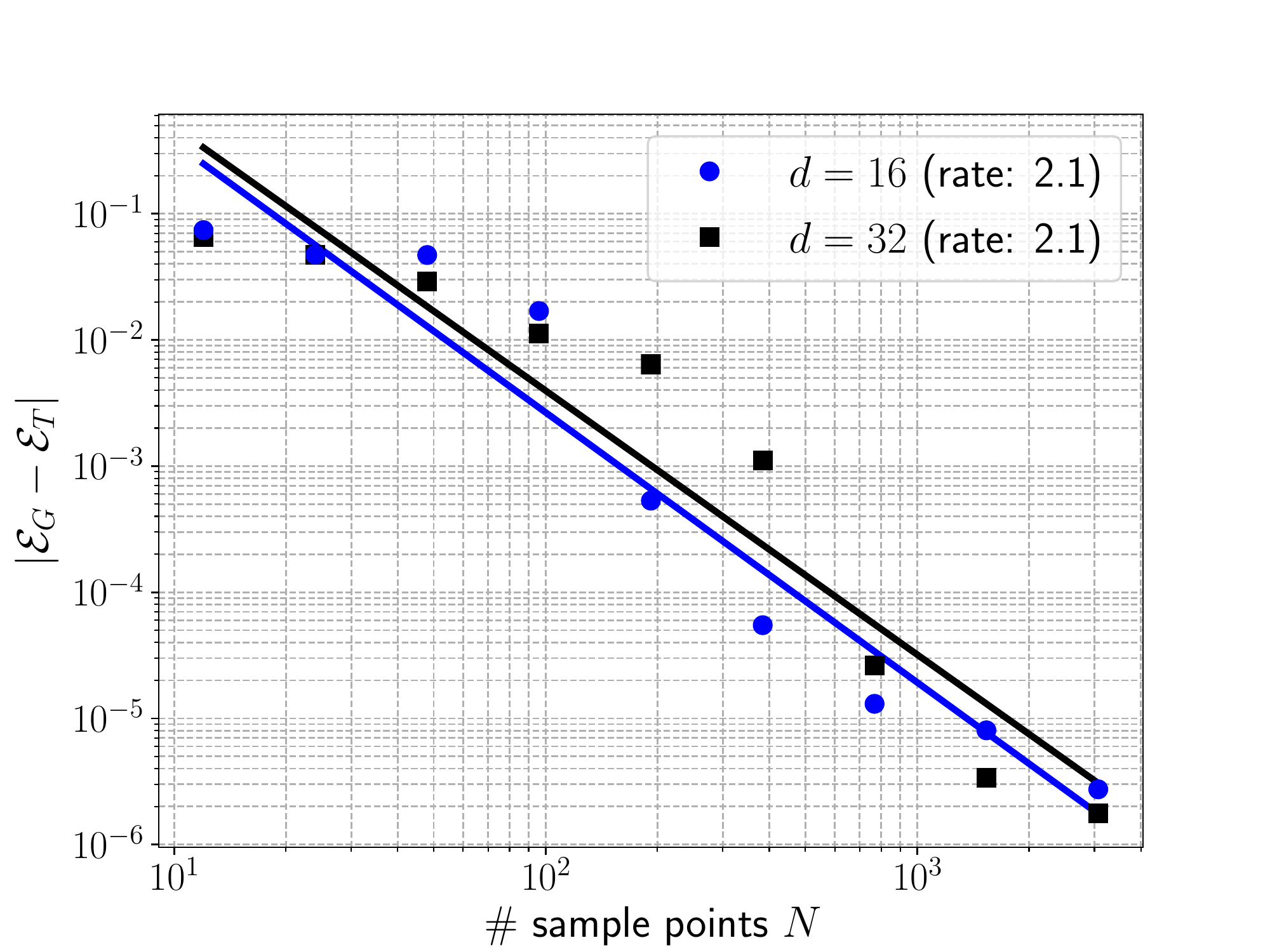}
\caption{Generalization gap $|\mathcal{E}_G - \mathcal{E}_T|$ for approximating the observable $g_h$ corresponding to the elliptic PDE \eqref{ellPDE} in 16 dimensions as well as in 32 dimensions.}
\label{fig:elliptic_pde_dim_comp}
\end{minipage}%
\hspace{0.01\textwidth}
\begin{minipage}[t]{.49\textwidth}
\includegraphics[width=1.\textwidth]{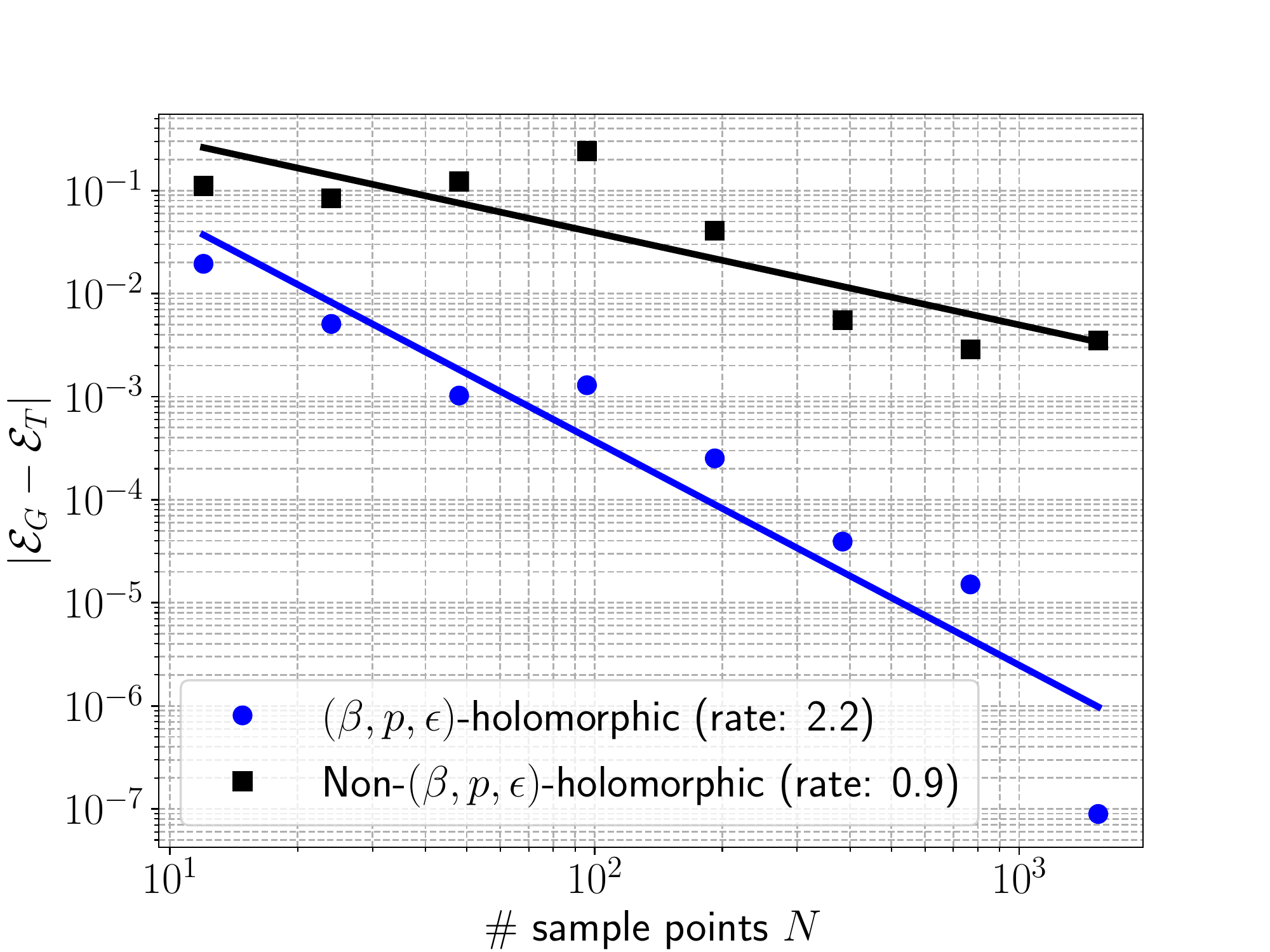}
\caption{Generalization gap $|\mathcal{E}_G - \mathcal{E}_T|$ using untrained DNNs based on the observable $g_h$ corresponding to the elliptic PDE \eqref{ellPDE} in 16 dimensions for both $(\bmbeta,p,\eps) $-holomorphic as well as non-$(\bmbeta,p,\eps) $-holomorphic networks.}
\label{fig:assumptions_check}
\end{minipage}
\end{figure}
\subsubsection{Holomorphy assumptions on the neural networks}
\label{sec:HolNN}
We examine the assumptions in Proposition \ref{prop:deepholo} which guarantee that a DNN is $(\bmbeta,p,\eps) $-holomorphic on polytubes in the following numerical experiment. As Proposition  \ref{prop:deepholo} does not rely on trained neural networks, we will focus on the case where the neural networks are untrained, i.e. the weights and biases are selected a priori and the generalization error \eqref{eq:gerr} is evaluated on an increasing set of \emph{test points}.  To this end, we generate a neural network with fixed width and depth and randomly generated weights. Note that the weights are chosen such that the neural network does not satisfy the assumptions of Proposition \ref{prop:deepholo}. We also construct a second DNN by clamping the weights of the first network into certain ranges such that the resulting network satisfies all assumptions of Proposition \ref{prop:deepholo} and thus is $ (\bmbeta,p,\eps) $-holomorphic on polytubes. \Fref{fig:assumptions_check} shows the generalization gap $|\mathcal{E}_G - \mathcal{E}_T|$ using the non-clamped (non-$ (\bmbeta,p,\eps) $-holomorphic) neural network as well as the clamped network ($ (\bmbeta,p,\eps) $-holomorphic). We see that in the case of the clamped neural network, a second order convergence is obtained, while the non-clamped network has a convergence rate of a bit less than 1. This supports the sufficiency of the assumptions of Proposition \ref{prop:deepholo} also numerically.

Despite this observation, we will not constrain the weights during training. Apart from the technical difficulty of doing so, we discover that most of the trained networks verified the assumption of  Proposition \ref{prop:deepholo} \emph{a posteriori}, possibly on account of the role played by the regularization term in the loss function \eqref{eq:lf}. 

\subsubsection{On the choice of activation functions}
\label{sec:ActFnct}
As the holomorphy of the DNN follows from holomorphy of the underlying activation function in \eqref{def:dnn}, 
we expect that the ReLU activation, which is not holomorphic at the origin, will lead to 
DNNs with worse convergence rates of the generalization gap 
than the DNNs with a holomorphic activation function such as the hyperbolic tangent. 
To investigate this, 
in \Fref{fig:relu_vs_tanh_elliptic}, we present the numerical generalization error and 
the numerically estimated training error for approximating the observable $g_h$ 
corresponding to the elliptic PDE \eqref{ellPDE} in  $16$ dimensions using tanh as well as ReLU as activations 
of the neural networks. 
We observe that using ReLU activations results in a slightly lower absolute generalization error than using tanh for small number of training samples. However, the rate of convergence using tanh activation appears to be close to $2.2$, while the convergence rate for using ReLU activation is close to $0.9$ and thus roughly one order lower than using tanh.  We conclude that for this example, the absolute generalization error is lower for tanh than for ReLU when the number of training samples is increased. \Fref{fig:relu_vs_tanh_elliptic_gap} shows the generalization gap for the same experiment.
We observe the same behavior as in \Fref{fig:relu_vs_tanh_elliptic}. 
This is to some extent expected, as the training error in \Fref{fig:relu_vs_tanh_elliptic} is almost constant for different numbers of training points for both tanh activation as well as for ReLU activation.
\begin{figure}[h!]
\centering
\begin{minipage}[t]{0.49\textwidth}	
\includegraphics[width=1.\textwidth]{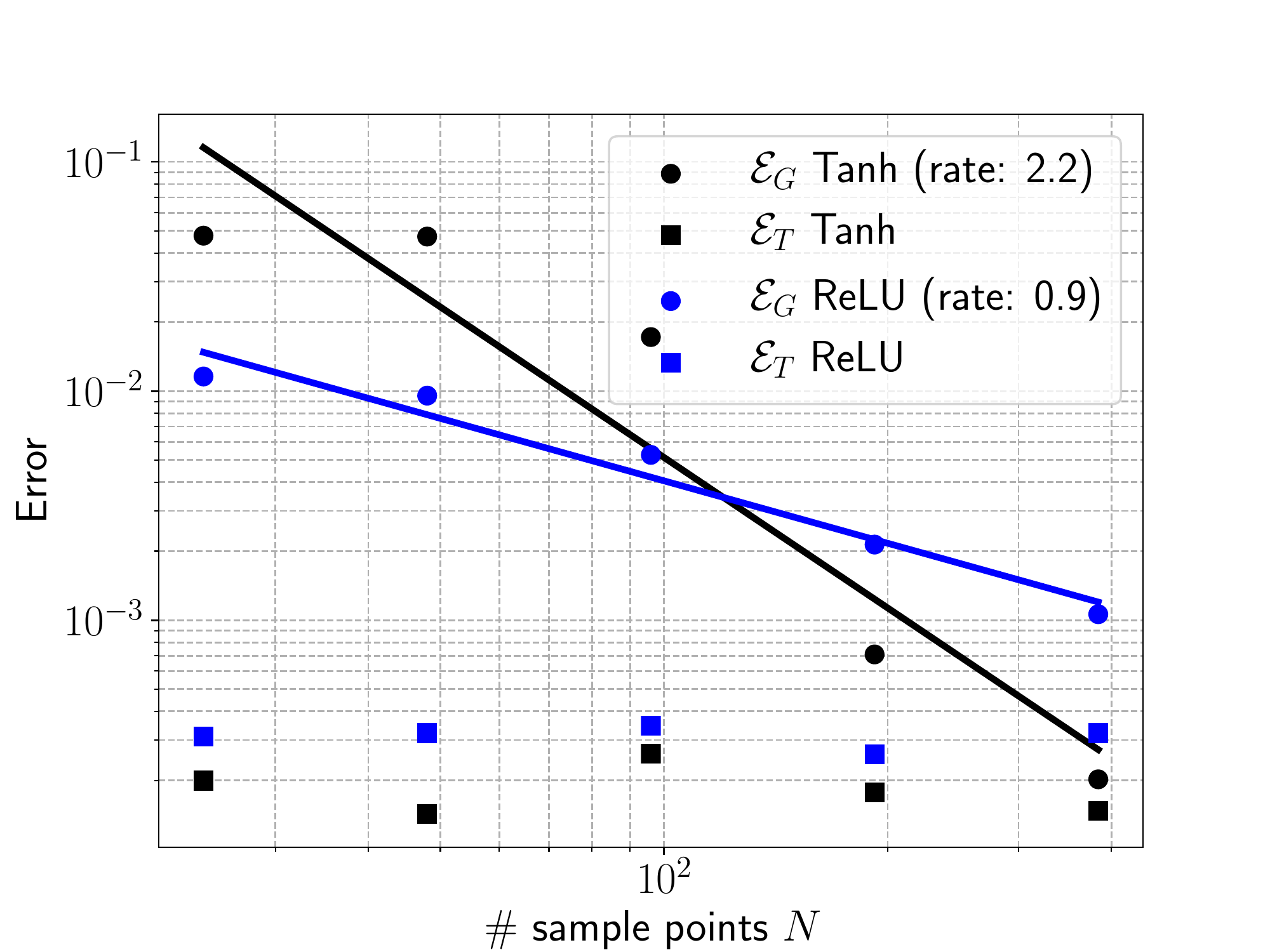}
\caption{Generalization error $\mathcal{E}_G$ together with the training error $\mathcal{E}_T$ for approximating the observable $g_h$ corresponding to the elliptic PDE \eqref{ellPDE} in 16 dimensions using tanh as well as ReLU as activations of the neural networks.}
\label{fig:relu_vs_tanh_elliptic}
\end{minipage}%
\hspace{0.01\textwidth}
\begin{minipage}[t]{.49\textwidth}
\includegraphics[width=1.\textwidth]{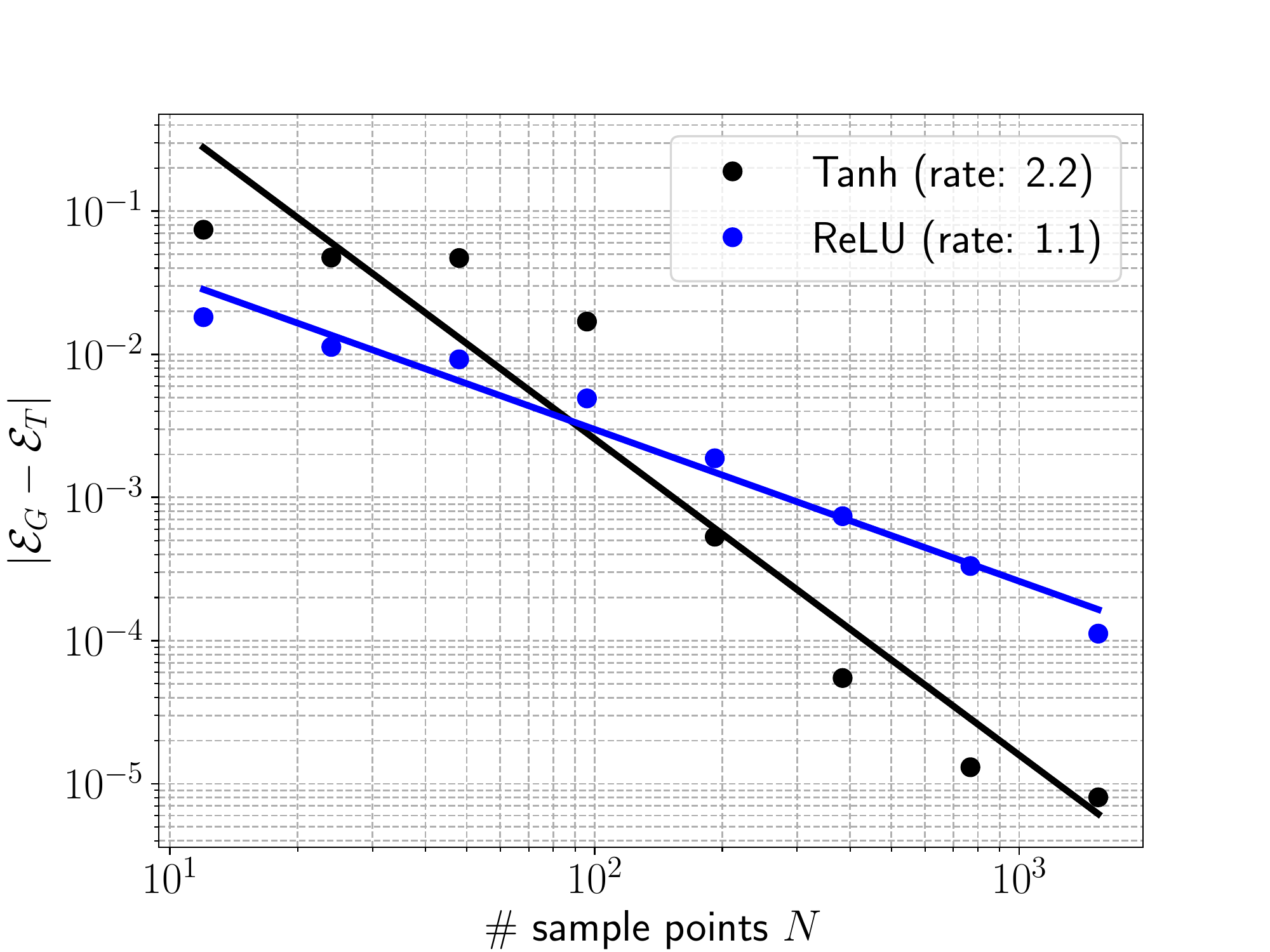}
\caption{
Generalization gap $|\mathcal{E}_G - \mathcal{E}_T|$ for approximating the 
observable $g_h$ corresponding to the elliptic PDE \eqref{ellPDE} 
 on a parameter domain of $16$ 
dimensions using tanh as well as ReLU as activations of the neural networks.}
\label{fig:relu_vs_tanh_elliptic_gap}
\end{minipage}
\end{figure}

\subsubsection{On the discretization of the training data}
\label{sec:disc}
As already mentioned in a previous section, in practice the training data is not provided by exact values of the DtO map but rather based on some approximations, for instance based on FEM approximations of the PDE. In the following, we conduct an experiment to analyze the effect of the approximation of the training data on the generalization of the DNNs numerically. To this end, we consider the observable  $g_h$ corresponding to the elliptic PDE \eqref{ellPDE} in 16 dimensions. We compute the training points of the observable based on three different FEM meshes, i.e. $g_{h_1}$ with $32768$ elements, $g_{h_2}$ with $131072$ elements and $g_{h_3}$ with $524288$ elements. Note that we use first order Lagrange elements in all cases and that the other experiments are all based on $g_{h_2}$.
\Fref{fig:disc_test} shows the generalization gap $|\mathcal{E}_G - \mathcal{E}_T|$ based on all three data sets corresponding to different meshes. The generalization error $\mathcal{E}_G$, 
for each set of training data, based on a different number of elements, 
is calculated using test points corresponding to different ground truths, i.e. $g_{h_i}$, for $i=1,2,3$. 
We can see that in all cases the convergence rate is almost the same and in particular of second order. 
We can conclude that as long as the FEM discretization is stable, 
the impact of the approximation error in the training data on the DL-HoQMC algorithm is negligible.

\begin{figure}[h!]
\centering
\begin{minipage}[t]{0.49\textwidth}	
\includegraphics[width=1.\textwidth]{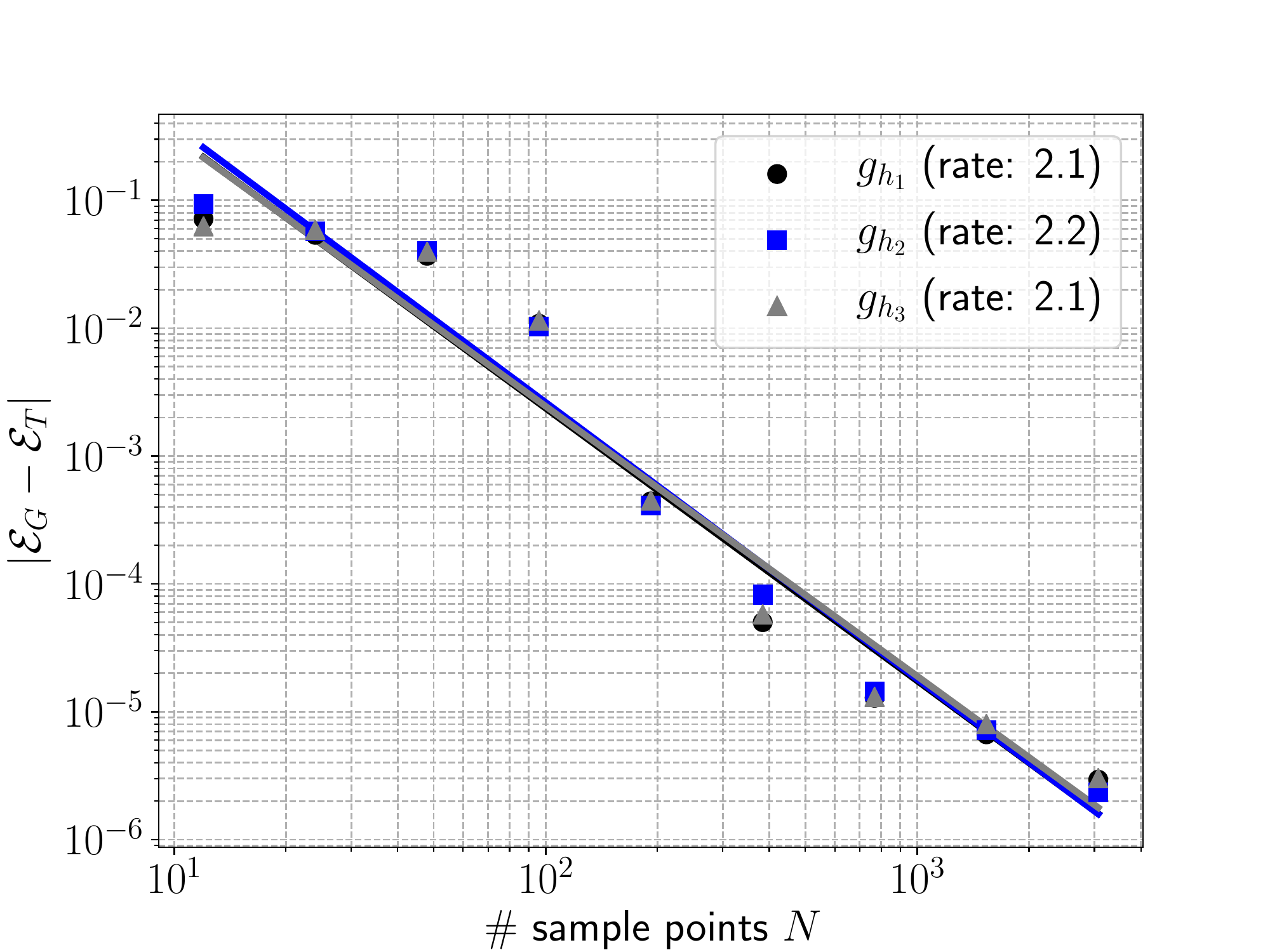}
\caption{Generalization gap $|\mathcal{E}_G-\mathcal{E}_T|$ for approximating observables $g_{h_i}$ corresponding to the elliptic PDE \eqref{ellPDE} in $16$ dimensions and based on three different FEM meshes.}
\label{fig:disc_test}
\end{minipage}%
\hspace{0.01\textwidth}
\begin{minipage}[t]{.49\textwidth}
\includegraphics[width=1.\textwidth]{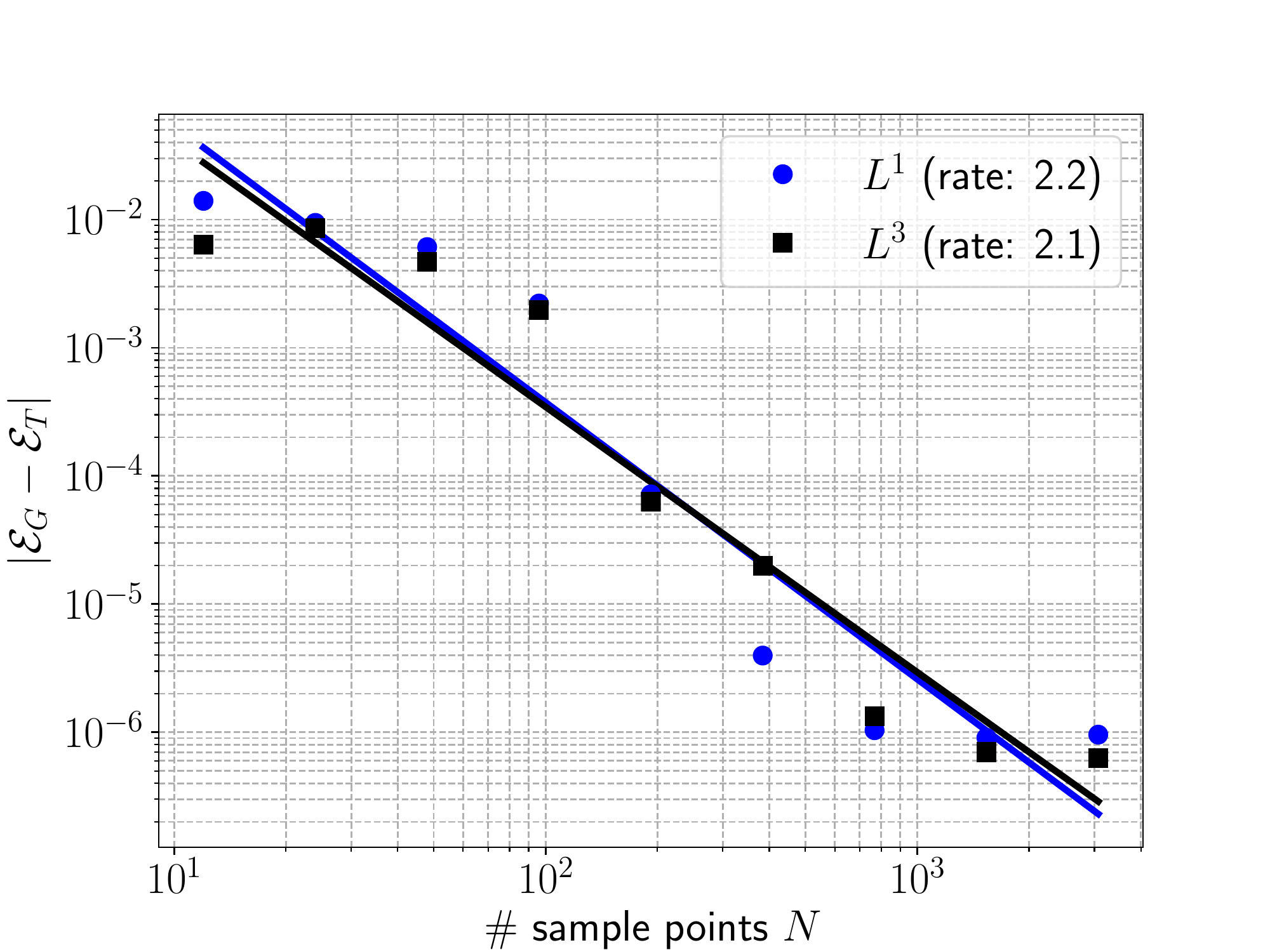}
\caption{
Generalization gap $|\mathcal{E}_G-\mathcal{E}_T|$ for approximating observable $g_{h}$ corresponding to the elliptic PDE \eqref{ellPDE} in $16$ dimensions using different $L^q$ norms in the cost function, i.e. $L^1$ and $L^3$.}
\label{fig:Lp_conditions}
\end{minipage}
\end{figure}
\subsubsection{On the holomorphy of the cost function}
\label{sec:HolomCstFn}
So far, we exclusively considered the case where $\calE_G$ and $\calE_T$ 
are based on the $L^2$ norm (or any $L^{2\mathbb{N}}$ norm) in order 
to ensure holomorphy of the integrand in \eqref{eq:gerr}. 
Hence, one question that naturally arises is if this condition 
is not only sufficient but also necessary. 
To test this numerically, 
we change the definition of $\mathcal{E}_G$  and $\mathcal{E}_T$ by using $L^q$ norms, 
where $q$ is an odd natural number. 
\Fref{fig:Lp_conditions} shows the computed generalization gap $|\mathcal{E}_G - \mathcal{E}_T|$ 
based on two different $L^q$ norms, i.e. $L^1$ and $L^3$.
We can see that a convergence of roughly second order is still obtained in both cases. 
Thus, this experiment suggests that the holomorphy condition 
on the cost function is sufficient but, apparently, 
not necessary in order to obtain second order convergence using the DL-HoQMC algorithm.
\subsection{ Linear Parabolic PDEs}
\label{sec:ParPDE}
For the next experiment, we consider a time-dependent problem 
defined by the following parametric parabolic PDE, 
which arises in the context of UQ for parabolic PDEs with uncertain initial data. 
For all $  y \in \left [ - \frac{1}{2}, \frac{1}{2} \right ]^{d} $, 
\begin{equation}\label{parPDE}
	\begin{cases}
	\partial_t u(x,t,y)-\div \left (\nabla u(x,t,y) \right ) = f(x,t) & x \in D, t \in [0,T] 
        \\
	u(x,t,y)= 0 & x \in \partial D, t \in [0,T]  \\
	u(x,0,y) = u_0(x,y) & x \in D
	\end{cases}
	.
\end{equation}
Here the divergence $ \div $ and the gradient $ \nabla $ are understood only with respect to the variables $ x $.
We use $ D = (0,1)^2 $ and the uncertain initial data $u_0$ is parametrized by
\begin{equation*}
u_0(x,y) = \exp\left (100\sum_{j=1}^dy_j\psi_j(x)\right ) - 1,
\end{equation*} 
where we use the same $\sin$ expansion for $\psi_j$ as in the elliptic case in the previous section. 
As for the source term $ f $, we select a moving localized source as
\begin{equation*}
f(x,t) = 100\exp(-20(x_1 - t)^2 - 20(x_2 - t)^2)
\end{equation*}
and for an observable we choose $ g\colon y\mapsto \dfrac{1}{|\tilde{D}|}{\displaystyle \int_{\tilde{D}} } u(\cdot,T,y)$. 
Again, we select $ \alpha = 2, \eta = 2.5, d = 16 $; 
then we run FEM simulations with $131072$ elements for each QMC sample point, 
here corresponding to EPL points, and final time $ T = 0.5 $. 
The time integration is by a backward Euler method, 
with constant time-step $ \Delta t = 10^{-3} $. We denote this approximation again as $g_h$.
\begin{figure}[h!]
\centering
\begin{minipage}[t]{.49\textwidth}
\includegraphics[width=1.\textwidth]{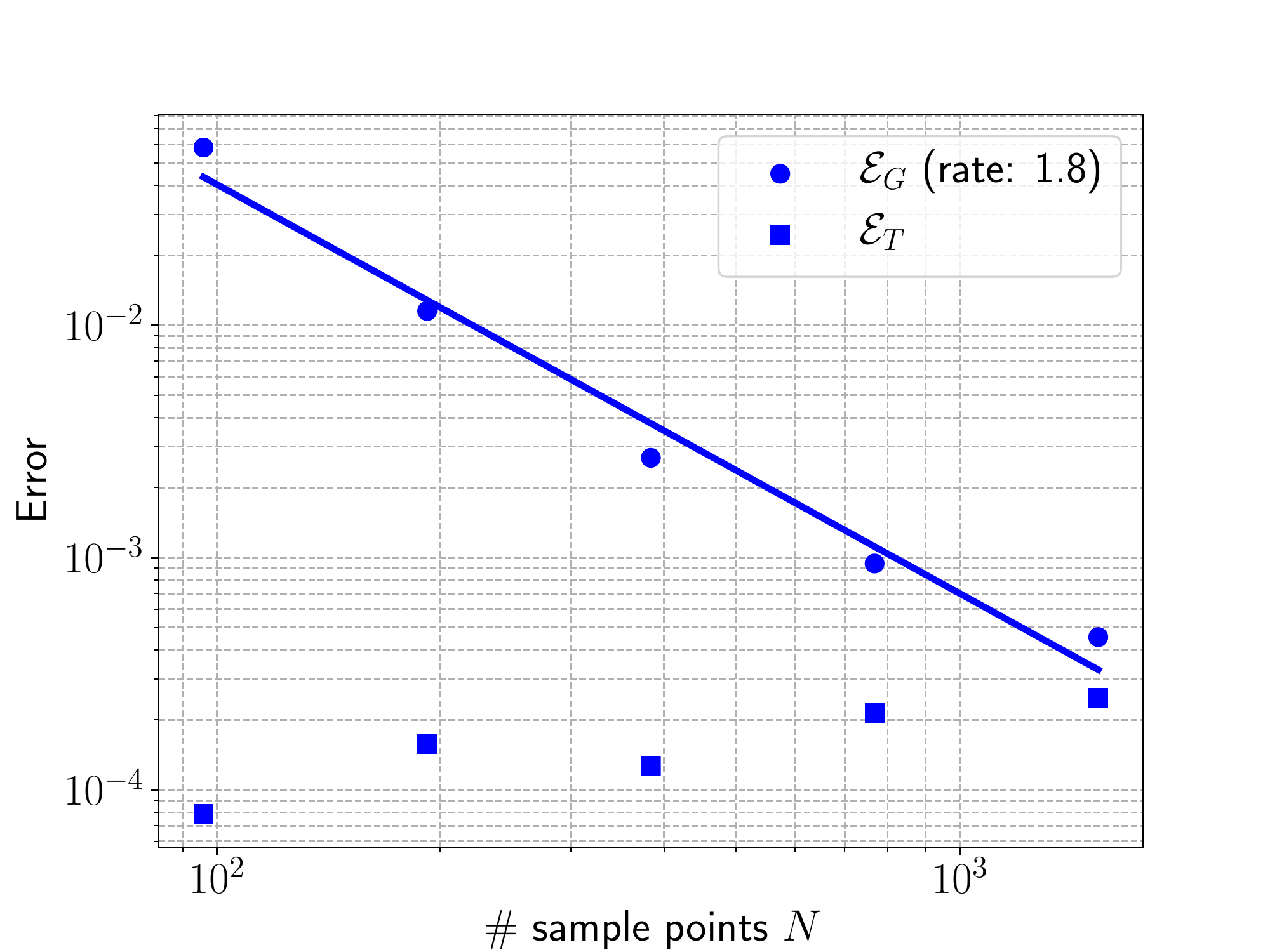}
\caption{Generalization error $\mathcal{E}_G$ for approximating the observable $g_h$ corresponding to the linear parabolic PDE \eqref{parPDE} in 16 dimensions together with the training error $\mathcal{E}_T$.}
\label{fig:parabolic_uq_init_data_gen}
\end{minipage}%
\hspace{0.01\textwidth}
\begin{minipage}[t]{0.49\textwidth}	
\includegraphics[width=1.\textwidth]{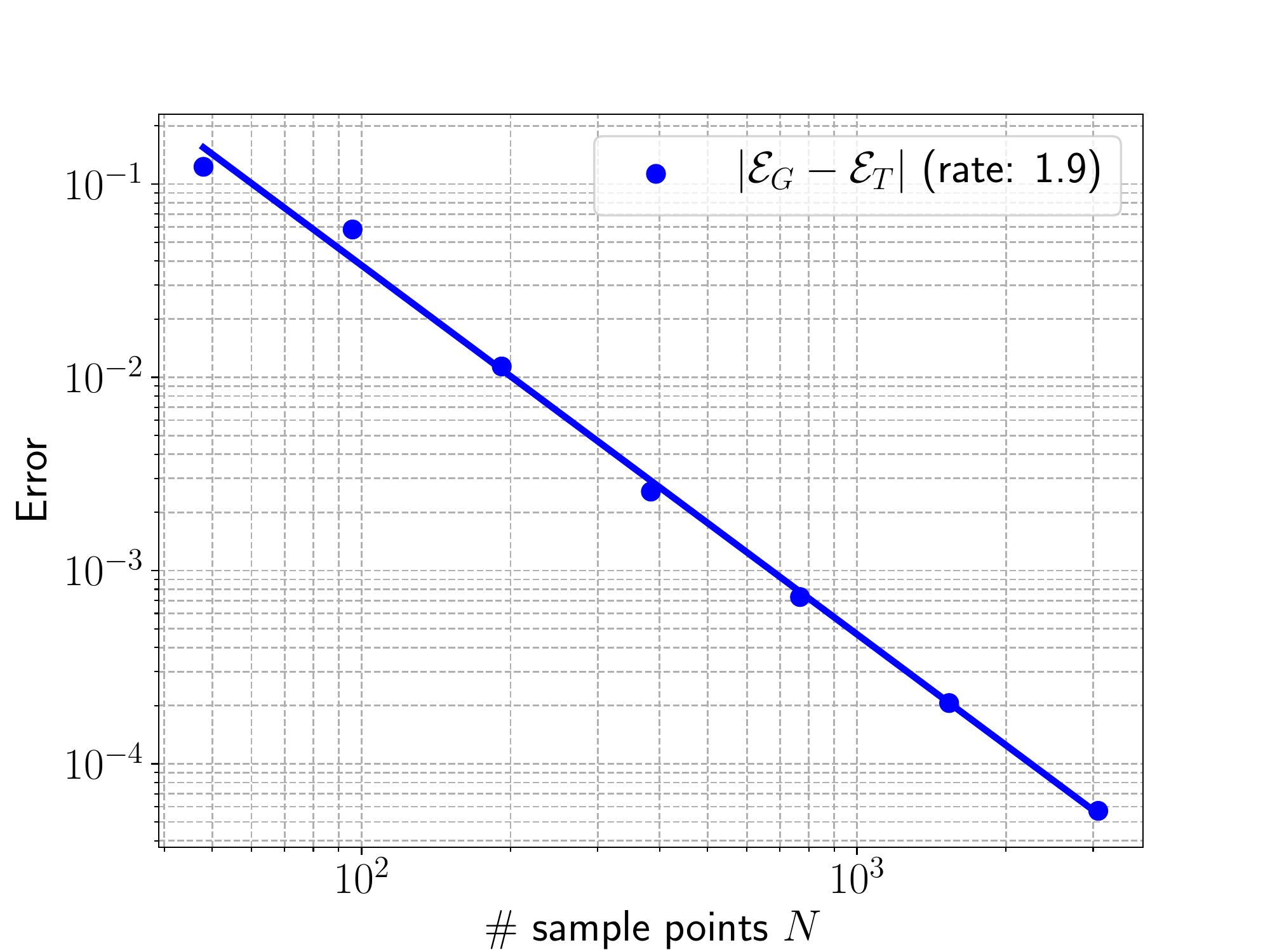}
\caption{Generalization gap $|\mathcal{E}_G - \mathcal{E}_T|$ for approximating the observable $g_h$ corresponding to the linear parabolic PDE \eqref{parPDE} in 16 dimensions.}
\label{fig:parabolic_uq_init_data_gap}
\end{minipage}
\end{figure}
\Fref{fig:parabolic_uq_init_data_gen} shows the generalization error $\mathcal{E}_G$ together with the training error $\mathcal{E}_T$ 
of the averaged trained ensemble for approximating the observable $g_h$ corresponding to the time-evolution equation \eqref{parPDE} 
for $d=16$ dimensions. 
We can see that the decay of the generalization error is approximately of second order. 
Additionally, \Fref{fig:parabolic_uq_init_data_gap} shows the generalization gap $|\mathcal{E}_G - \mathcal{E}_T|$, which also decays with a rate of almost $2$. 
This is again in accordance with our theoretical findings and thus demonstrates that the proposed algorithm can successfully be applied to UQ for time-dependent PDEs.

Note that compared to the previous experiments, the numerical range of the observable is rather small in most of the parameter domain $ \left [ - \frac{1}{2}, \frac{1}{2} \right ]^{d} $, with the exception of values of $ y $ towards the corner $ (\frac{1}{2},\ldots,\frac{1}{2}) $, where the range is much larger. 
The presence of such localized strong variations in the samples makes the training particularly difficult. Hence, in order to allow the DNNs to learn such outliers, we reduce the regularization parameter $\lambda$ in \Tref{tab:ensemble_params} to $10^{-7},10^{-8},10^{-9}$. Additionally, we increase the number of epochs (equivalent to training iterations in our full-batch mode) to 100k.
\section{Discussion}
\label{sec:5}
A diverse set of problems in scientific computing involving PDEs, such as
uncertainty quantification (UQ), (Bayesian) inverse problems, optimal control and design 
are of the \emph{many query} type. 
I.e., their numerical solution requires a large number of calls to some underlying numerical PDE solver. 
As PDE solvers, particularly in multiple space dimensions, could be expensive, 
the numerical solution of such many query problems can be prohibitively expensive. 
Hence, the design of efficient and accurate surrogate models 
is of great importance as they can make such many query problems 
computationally tractable in engineering applications.

In this article, we propose a surrogate algorithm  
based on deep neural networks, to approximate observables of interest, 
on solution families of PDE models. 
The key novelty of our algorithm \emph{lies in the use of  deterministic families of training points, 
defined by polynomial lattices, that arise in the context of high-order Quasi-Monte Carlo (HoQMC) integration methods}. 
In particular, 
we employ training points that correspond to quadrature points defined by 
extrapolated polynomial lattice (EPL) points (\cite{DGY19,DLS20_902})
or 
interlaced polynomial lattice (IPL) points  (\cite{DKLNS14} and the references therein).
The resulting DL-HoQMC algorithm possesses the following attribute; 
as long as the underlying map (observable) is holomorphic in a precise sense, defined in Section \ref{sec:3},
and the underlying deep neural networks are such that the activation function is similarly holomorphic 
as well as the weights of the neural network satisfy the conditions of Proposition \ref{prop:deepholo}, 
we prove that the resulting generalization gap (i.e., the difference between generalization and training errors),
\begin{itemize}
\item 
decays quadratically 
(and, given sufficiently small summability exponent $p\in (0,1)$ and sufficient large QMC integration order,
at arbitrary high-order) 
with respect to the number of training points,
\item 
the rate of convergence is \emph{independent of the dimension $d$ of the parameter domain}.
\end{itemize}
Thus, at least for a (large) class of maps, 
the proposed algorithm has a significantly higher rate of decay of the generalization 
gap (in terms of the size of training set) 
than standard deep learning algorithms that use random training points, 
while at the same time overcoming the \emph{curse of dimensionality}. 
 This is afforded by suitable sparsity in the DtO map, 
     as ensured here via quantified parametric holomorphy.
This removes a major bottleneck in the use of deep learning algorithms 
for regression problems in scientific computing, 
where the use of i.i.d random training points requires 
large computational resources on account of the slow rate of convergence.
See, e.g., \cite{MR1} and references therein. 

We present numerical experiments, 
involving model problems for both elliptic and parabolic PDEs, 
that validate the proposed theory and demonstrate the ability of the 
DL-HoQMC algorithm to approximates observables of PDEs in very high dimensions, efficiently.  

Hence, 
the proposed algorithm promises to provide efficient DNN surrogates for  parametric PDEs 
with high dimensional state- and / or parameter spaces.
Nevertheless, 
it is important to point out the following caveats:
\begin{itemize}
\item 
The estimates we prove in Theorem \ref{thm:qmcgap} are on the generalization gap 
and we do not attempt to estimate the training error in any way. 
However, this is standard practice in machine learning \cite{MLbook} as estimating the training error 
that arises from using stochastic gradient descent for a (highly) non-convex 
very high dimensional optimization problem is quite challenging. 
\item 
The estimate on the generalization gap requires holomorphy 
of the DtO map and of the DNN emulating this DtO map.
In Proposition \ref{prop:deepholo}, 
we provide sufficient conditions on the weights to the network in order to ensure this holomorphy. 
However, 
in practice and as pointed out in Section \ref{sec:4}, 
we do not explicitly ensure that the trained weights satisfy these bounds. 
These bounds are verified \emph{a posteriori} imply a subtle role for 
regularization of the loss function \eqref{eq:lf} that needs to further elucidated. 
\item 
We measured the generalization gap in integral norms with even $q\in 2{\mathbb N}$,
in order to allow for analytic continuation of the parametric integrand, 
which was a necessary ingredient to allow for QMC integration.
Numerical experiments indicated, however, a similar generalization gap decay
also for $q=1,3$, indicating that the condition $q\in 2{\mathbb N}$ is
not necessary.
\item 
In the numerical examples reported in Section \ref{sec:4},
we considered linear, elliptic and parabolic partial differential
equations, subject to either affine-parametric uncertain input data, or 
subject to holomorphic maps (specifically, $\exp( )$) of such inputs.
It was proved in e.g. \cite{CCS15} for nonlinear, 
parametric holomorphic operator equations that such inputs, 
with summability conditions of the sequence $\bmbeta$ following from
(assumed) decay conditions for the elements of the sequence $\{ \psi_j \}_{j\geq 1}$,
imply $(\bmbeta,p,\eps)$-holomorphy of the parametric solution manifolds.
Further examples of $(\bmbeta,p,\eps)$-holomorphic DtO maps include 
time-harmonic electromagnetic scattering (e.g.\ \cite{AJSZ20_2734})
in parametric scatterers, 
viscous, incompressible fluids in uncertain geometries (e.g.\ \cite{CSZ18_2319}),
boundary integral equations on parametric boundaries (e.g.\ \cite{HS19_847}),
and 
parametric, dynamical systems described by large systems of initial-value ODEs
(e.g.\ \cite{RegDedAQ_MORODEDNN} and, for a proof of 
parametric holomorphy of solution manifolds, \cite{HS13_1085}),
and \cite{HSZ20_875} for DtO maps for Bayesian Inverse Problems for PDEs.

The presently developed results being based only on quantified, parametric
holomorphy on the DtO maps will be directly applicable also to these
settings.
\end{itemize}
Finally, we point out that the DL-HoQMC algorithm 
can be used to accelerate the computations for UQ \cite{LMR1} and 
PDE constrained optimization \cite{LMRP1}, 
among other many-query problems for DtO maps of systems governed by PDEs.

\section*{Acknowledgements.} The research of SM and TKR was partially supported by European Research Council Consolidator grant ERCCoG 770880: COMANFLO.

{\small
\bibliographystyle{plain}
\bibliography{References}
}
\end{document}